\documentclass[fontsize=10pt]{scrreprt}

\usepackage[square,numbers,nonamebreak]{natbib}
\usepackage{setspace}
\usepackage{enumitem}
\usepackage[english]{babel}

\usepackage{chngcntr}

\usepackage[hang]{footmisc}
 at 11truept
\font\ssc=pplrc9d at 11 truept
\newcommand\qedbox{$\rlap{$\sqcap$}\sqcup$}

\usepackage[hmarginratio=1:1, bottom=2cm, top=2.5cm]{geometry}

\usepackage[automark]{scrlayer-scrpage}
\cohead{\textnormal{Skew Braces from a model-theoretic point of view}}
\lehead{\pagemark}
\rohead{\pagemark}

\let\ceheadL\cehead
\renewcommand\cehead[1]{
\ceheadL{\textnormal{#1}}
}

\cehead{Maria Ferrara --- Moreno Invitti --- Marco Trombetti --- Frank O. Wagner}
\lefoot{}
\rofoot{}

\usepackage{graphicx}
\usepackage[usenames,dvipsnames,svgnames,table]{xcolor}
\definecolor{Maroon}{cmyk}{0, 0.87, 0.68, 0.32}
\definecolor{RoyalBlue2}{cmyk}{80,100,0,0.1}

\newcommand\auths[1]{\large \textsc{\textcolor{Maroon}{#1}}\setstretch{1.2}}
\newcommand\titl[1]{\center \linespread{1.1}\color{RoyalBlue2}\Large\textbf{ #1}\color{black}\bigskip} 
\renewcommand\abstract[1]{
\begin{center}
{\textbf{Abstract}}
\end{center}
{
\linespread{1.1}\fontsize{9pt}{-10pt}\selectfont #1}}

\usepackage[utf8]{inputenc}
\usepackage[sc,osf]{mathpazo}
\usepackage[euler-digits]{eulervm}
\usepackage[T1]{fontenc}
\usepackage{mathtools}

\DeclareSymbolFont{operators}{\encodingdefault}{ppl}{m}{n}
\DeclareMathAlphabet{\mathbf}{\encodingdefault}{ppl}{bx}{n}
\DeclareMathAlphabet{\mathit}{\encodingdefault}{ppl}{m}{it}

\usepackage{amsfonts}
\usepackage{amsmath}
\usepackage{ragged2e}
\usepackage{titlesec}
\usepackage{amssymb}

\renewcommand{\thesection}{\arabic{section}}
 
\titleformat{\section}{\medskip\bigskip\normalfont\Large\bf}{\thesection}{0.5em}{}
\titleformat{\subsection}{\smallskip\bigskip\normalfont\large\bf}{\thesubsection}{0.5em}{}

\usepackage{amsthm}
\newtheoremstyle{dotless}{}{}{\itshape}{}{\bfseries}{}{1em}{}

\theoremstyle{dotless}
\newtheorem{theo}{Theorem}
\newtheorem{prop}[theo]{Proposition}
\newtheorem{lem}[theo]{Lemma}
\newtheorem{cor}[theo]{Corollary}
\newtheorem{rem}[theo]{Remark}

\newtheorem{fact}[theo]{Fact}
\renewenvironment{proof}{\smallbreak\noindent {\sc Proof \;---\;}}{\hfill\qedbox\smallskip}

\counterwithout{theo}{section}
\numberwithin{theo}{section}

\DeclareOldFontCommand{\rm}{\normalfont\rmfamily}{\mathrm}
\DeclareOldFontCommand{\sf}{\normalfont\sffamily}{\mathsf}
\DeclareOldFontCommand{\tt}{\normalfont\ttfamily}{\mathtt}
\DeclareOldFontCommand{\bf}{\normalfont\bfseries}{\mathbf}
\DeclareOldFontCommand{\it}{\normalfont\itshape}{\mathit}
\DeclareOldFontCommand{\sl}{\normalfont\slshape}{\@nomath\sl}
\DeclareOldFontCommand{\sc}{\normalfont\scshape}{\@nomath\sc}

\numberwithin{theo}{section}

\usepackage{comment}

\usepackage{soul}

\def\bar{\overline}
\def\tp{\mbox{\rm tp}}
\def\Ann{\operatorname{Ann}}
\def\Soc{\operatorname{Soc}}
\def\Ker{\operatorname{Ker}}
\def\Fixl{\operatorname{Fix}^l}
\def\Fixr{\operatorname{Fix}^r}
\def\Stab{\operatorname{Stab}}
\def\dcl{\operatorname{dcl}}
\def\mc{$\mathfrak M$c}
\def\Kid{\operatorname{Kid}}

\begin{document}

\titl{Skew Braces from a model-theoretic point of view
\footnote{The first and third authors are supported by GNSAGA (INdAM) and are members of the non-profit association Advances in Group Theory and Applications (www.advgrouptheory.com). The first author was supported by the project COSYMA -- CUP B26G21000070005 and partially supported by the project COMBINE of VALERE: VAnviteLli pEr la RicErca of the University of Campania Luigi Vanvitelli. Funded by the European Union - Next Generation EU, Missione 4 Componente 1 CUP B53D23009410006, PRIN 2022- 2022PSTWLB - Group Theory and Applications. The fourth author would like to acknowledge the support of the Science Committee of the Ministry of Science and Higher Education of the Republic of Kazakhstan (Grant No.\ AP19677451).}}

\auths{Maria Ferrara --- Moreno Invitti --- Marco Trombetti --- Frank O. Wagner}

\thispagestyle{empty}
\justify\noindent
\setstretch{0.3}
\abstract{Skew braces are one of the main algebraic tools controlling the structure of a non-degenerate bijective set-theoretic solution of the Yang--Baxter equation. The aim of this paper is to study model-theoretically tame skew braces, with particular attention to the notions of solubility and nilpotency.}

\setstretch{2.1}
\noindent
{\fontsize{10pt}{-10pt}\selectfont {\it Mathematics Subject Classification (2020)}: 03C45, 16T25}\\[-0.8cm]

\noindent 
\fontsize{10pt}{-10pt}\selectfont  {\it Keywords}: stable theory; $\omega$-categorical theory; skew brace; Yang–Baxter equation\\[-0.8cm]

\setstretch{1.1}
\fontsize{11pt}{12pt}\selectfont
\section{Introduction}

\noindent The Yang--Baxter equation (YBE) is a fundamental equation of statistical mechanics that arose from the independent studies of the physicists Chen-Ning Yang \cite{Yang} and Rodney Baxter \cite{Baxter}. This equation plays a relevant role in many areas of mathematics such as knot theory, braid theory, operator theory, Hopf algebras, quantum groups, $3$-manifolds and the monodromy of differential equations. 

A {\it solution} to the YBE is a pair $(V,R)$, where $V$ is a vector space and $R$ is a linear map $R:\, V\otimes V\longrightarrow V\otimes V$ such that $$(R\otimes\operatorname{id})(\operatorname{id}\otimes R)(R\otimes\operatorname{id})=(\operatorname{id}\otimes R)(R\otimes\operatorname{id})(\operatorname{id}\otimes R).$$ We are currently far from being able to provide a full classification of the solutions to the YBE, but, in recent years, there has been some progress concerning the so-called {\it set-theoretic} (or, {\it combinatorial}) solutions of the YBE, \hbox{i.e.} those solutions $(V,R)$ such that~$R$ is induced by linear extensions of bijective maps $$r:\, X\times X\longrightarrow X\times X,$$ where $X$ is a basis of $V$ (see \cite{Drinfeld}). This is mainly because the construction of set-theoretic solutions can sometimes be based on the use of (associative and non-associa\-tive) algebraic structures (see \cite{Rump2}). Among these structures, skew braces hold a prominent position (see for example \cite{OriginalBrace},\cite{Vendramin}). A skew brace is essentially a set endowed with two group structures linked together with a ``distributivity-like'' relation (see next section for the precise definitions).

The aim of this paper is to study the first-order theory of skew braces, in particular under model-theoretic tameness conditions.

The layout of the paper is as follows. In Section \ref{preliminaries} we recall some of the basic definitions and results concerning skew braces, as well as the model-theoretic background. In Section \ref{stability} we introduce the chain conditions on subgroups which hold under various model-theoretic tameness conditions, and also apply both to the additive and multiplicative group of a skew brace. We introduce the connected component and the $\Phi$-component and show that they define ideals. In Section \ref{solubility} we study the various notions of solubility for a skew brace, and show that a soluble skew sub-brace of a stable skew brace is contained in a type-definable one. In Section \ref{nilpotency} we study nilpotent skew braces. The main results here are that an annihilator-nilpotent, socle-nilpotent of right nilpotent skew sub-brace of a stable skew brace is contained in a definable one of the same class; for a left nilpotent skew sub-brace we merely find an enveloping type-definable one, again of the same class. In Section \ref{omega-cat} we analyze locally finite and $\omega$-categorical skew braces, and show that a stable $\omega$-categorical skew brace is virtually left nilpotent of nilpotent type, and an $\omega$-stable $\omega$-categorical skew brace is abelian and trivial. Finally, in Section \ref{local-nilpotency} we study $\pi$-nilpotency in skew braces of nilpotent type.

\section{Preliminaries}\label{preliminaries}
\subsection{Skew braces}
Let $B$ be a set. If $(B,+)$ and $(B,\circ)$ are groups, then the triple $(B,+,\circ)$ is a {\it skew \textnormal(left\textnormal) brace} if the skew (left) distributive property $$a\circ(b+c)=a\circ b-a+a\circ c$$ holds for all $a,b,c\in B$. Now, let $(B,+,\circ)$ be a skew brace. We refer to $(B,+)$ as the \emph{additive group} of $B$ and to $(B,\circ)$ as the \emph{multiplicative group} of $B$. We denote by~$0$ the identity of $(B,+)$, by $1$ the identity of $(B,\circ)$, and  by $-a$ and $a^{-1}$ the inverses of~$a$ in~$(B,+)$ and $(B,\circ)$, respectively. The skew distributive property implies $0=1$. If $a,b\in B$, we employ the notation $[a,b]_+$ and $[a,b]_\circ$ to denote respectively the commutator in $(B,+)$ and $(B,\circ)$ of $a$ and $b$ (notice that our convention for commutators in a group $(G,\cdot)$ is the following one: $[x,y]=xyx^{-1}y^{-1}$); similar expedients are used for other natural group theoretic concepts. 
Note that we do not suppose addition to be commutative.

For any $b\in B$ we denote by $\lambda_b$ the map $x\mapsto -b+b\circ x$. It follows from left distributivity that $\lambda_b$ is an additive automorphism of $B$. The map $b\mapsto\lambda_b$ is a group homomorphism from $(B,\circ)$ to $\operatorname{Aut}(B,+)$ and the following relations hold $$a+b=a\circ\lambda_a^{-1}(b),\quad a\circ b=a+\lambda_a(b),\quad -a=\lambda_a\big(a^{-1}\big).$$

In analogy with ring theory, a third relevant (non-necessarily associative) operation in skew braces is defined as follows $$a\ast b=\lambda_a(b)-b=-a+a\circ b-b$$ and one easily checks that it satisfies the relations 
\[
\begin{array}{c}
a \ast (b + c) = a \ast b + b + a \ast c - b,\\[0.2cm]
(a+b)\ast c= a\ast\big(\lambda_a^{-1}(b)\ast c)+\lambda_a^{-1}(b)\ast c+a\ast c\quad\textnormal{and}\\[0.2cm]
(a \circ b) \ast c=a \ast (b \ast c) + b \ast c + a \ast c,
\end{array}
\] 
for all $a,b,c\in B$. Taking into account $G=(B,+)\rtimes (B,\circ)$, an easy computation shows that the $\ast$-operation corresponds to a commutator of type \[
\begin{array}{c}
\big[(0,a),(b,0)\big]=(a\ast b,0),
\end{array}
\] for all $a,b\in B$.

Let $\mathfrak{X}$ be a class of groups. If the additive group $(B,+)$ is in $\mathfrak{X}$, we call $B$ a (left) skew brace of type $\mathfrak{X}$. In case $\mathfrak{X}=\mathfrak{A}$ is the class of all abelian groups, we also say that $B$ is a (left) brace.

We shall call a skew brace {\it trivial} if addition equals multiplication, and {\it almost trivial} if addition is opposite multiplication. In this way, every group can be considered a trivial or an almost trivial skew brace.

A {\it left ideal} of a skew brace $B$ is an additive subgroup $I$ such that $\lambda_a(I)\subseteq I$ for all~\hbox{$a\in B$;} this is equivalent to $B\ast I\subseteq I$, so $I$ is also a subgroup of $(B,\circ)$ and~\hbox{$a\circ I\subseteq a+I$} for all $a\in B$ (considering $a^{-1}$ we actually have equality $a\circ I=a+I$). A left ideal $I$ is {\em strong} if it is additively normal. An {\it ideal} is a strong left ideal that is also multiplicatively normal (or equivalently such that $I\ast B\subseteq I$); in this case $B/I$ with the induced operations is again a skew brace. It is easy to see that the sum of two (strong) (left) ideals is again a (strong) (left) ideal.

The ideal generated by an element $b\in B$ will be denoted by $(b)$, and $I\trianglelefteq B$ denotes that $I$ is an ideal in $B$. An ideal will be called {\it trivial} if it is trivial as a skew brace.
A skew brace is {\it simple} if it has no proper non-zero ideals. A skew sub-brace $I\le B$ is a sub-(left)(strong) ideal if there is a sequence
$$B=I_0\ge I_1\ge\cdots\ge I_n=I$$
such that $I_{i+1}$ is a (strong)(left) ideal in $I_i$ for all $i<n$.

If $C$ is another skew brace, a {\it homomorphism} (resp.\ {\it isomorphism}) from $B$ to~$C$ is as usual a function (resp.\ {\it bijective function}) $\varphi:B\longrightarrow C$ preserving addition and multiplication; an {\it automorphism} of $B$ is just an isomorphism from $B$ to $B$. The kernels of brace homomorphisms are precisely the ideals, and the homomorphism theorems hold for braces.

An important group associated to a skew brace $B$ is its {\em design group} $G(B)=(B,+)\rtimes_\lambda(B,\circ)$, where $$(a,b)(c,d)=(a + \lambda_b(c),b\circ d)$$
for all $a,b,c,d\in B$. Clearly, a subgroup $A\le(B,+)$ (considered as a subgroup of $G$) is a strong left ideal iff it is normal in $G$.

\begin{lem}\label{idealnormal} Let $B$ be a skew brace. Let $I\trianglelefteq B$ be an ideal in $B$. Then $G(I)$ is a normal subgroup of $G$. If $J\le B$ is a skew sub-brace, then $G(I)\,G(J)=G(I+J)$.\end{lem}
Recall that since $I$ is an ideal, $I+J=J+I=J\circ I=I\circ J$.
\begin{proof} Since $I$ is an ideal in $B$, clearly $(I,+)$ is normal in $G(B)$ and $(I,\circ)$ is normal in $(B,\circ)$. Moreover, the $\lambda$-action of $(I,\circ)$ on $(B/I,+)$ is trivial, which means that $G(I)$ is normal in $G(B)$.
	
Now consider $(a,b)\in G(I)$ and $(a',b')\in G(J)$. Then 
$$(a,b)(a'b')=\big(a+\lambda_b(a'),b\circ b'\big)\in G(I+J).$$
Conversely, if $a,b\in I$ and $a',b'\in J$, then
$$(a+a',b\circ b')=(a,b)\,\big(\lambda_{b^{-1}}(a'),b'\big)\in G(I)\,G(J).$$
Equality follows.\end{proof}

The {\it socle} of $B$ is defined as $\Soc(B)=\Ker(\lambda)\cap Z^+(B)$ and the {\it annihilator} of $B$ is defined as $\Ann(B)=\Soc(B)\cap Z^\circ(B)$, where $Z^+(B)$ and $Z^\circ(B)$ denote the additive and multiplicative centre. More generally, we define the {\em upper annihilator series} by
$$\Ann_0(B)=\{0\},\quad\mbox{and}\quad\Ann_{n+1}(B)=\{b\in B:b\ast B,[b,B]_+,[b,B]_\circ\subseteq \Ann_n(B)\},$$
and the {\em upper socle series} by
$$\Soc_0(B)=\{0\},\quad\mbox{and}\quad
\Soc_{n+1}(B)=\{b\in B:b\ast B,[b,B]_+\subseteq \Soc_n(B)\}.$$
All of the above are ideals in $B$. We call $B$ {\em annihilator nilpotent of class $n$} if $\Ann_n(B)=B>\Ann_{n-1}(B)$, and {\em socle nilpotent of class $n$} if $\Soc_n(B)=B>\Soc_{n-1}(B)$.

We define two descending series $B^{(n)}$ and $B^n$ as follows:
$$B^{(1)}=B^1=B,\quad B^{(n+1)}=\langle B^{(n)}\ast B\rangle_+,\quad\mbox{and}\quad B^{n+1}=\langle B\ast B^n\rangle_+.$$
While the $B^{(n)}$ are ideals, in general the $B^n$ are only left ideals in $B$, and $B^{n+1}$ is an ideal in $B^n$. We call $B$ {\em right nilpotent of class $n$} if $B^{(n+1)}=\{0\}$, and {\em left nilpotent of class $n$} if $B^{n+1}=\{0\}$. Then $B$ is socle nilpotent iff $B$ is right nilpotent of nilpotent type \cite[Lemma 2.16]{Cedo}.

We have not been able to find the following lemma in the literature.

\begin{lem} Let $B$ be a skew brace. Then $\ker\lambda$ is a trivial skew sub-brace.\end{lem}
\begin{proof} Closure under product and multiplicative inverse are obvious. So let $a,b\in\ker\lambda$. Then 
$$-a+b=-a+a\circ a^{-1}\circ b=\lambda_a(a^{-1}\circ b)=a^{-1}\circ b\in\ker\lambda.$$
It follows that $\ker\lambda$ is a skew sub-brace. Clearly it is trivial.
\end{proof}

\subsection{Model theory}
Let $\mathcal{L}=\{+,\circ,\,^{-1},-,0\}$ be the first-order language of skew braces. In what follows, unless mentioned otherwise, a {\it formula} is a formula in the language $\mathcal{L}$.
Let $B$ be a skew brace. Then $\operatorname{Th}(B)=\{\varphi\,:\, B\models\varphi\}$ denotes the first-order theory of $B$. If $\bar b$ is a tuple of elements from $B$ (denoted $\bar b\in B$ by abuse of notation), a subset $X$ of $B^n$ is {\it $\bar b$-definable} if $X=\big\{\bar a\in B^n\,:\, B\models\varphi(\bar a,\bar b)\big\}$ for some formula $\varphi(\bar x,\bar y)$; we may sometimes omit $\bar b$ if the exact parameters are not important. If $\bar b=\emptyset$, the set $X$ is {\it parameter-free definable} or {\it $\emptyset$-definable}; if we want to emphasize the formula $\varphi(\bar x,\bar y)$ used, we call $X$ $\varphi$-definable. We shall call a function {\it definable} if its graph is definable. For subsets $X,Y\subseteq B$ we put $X^{-1}=\{x^{-1}:x\in X\}$ and $X\lozenge Y=\{x\lozenge y:x\in X,\,y\in Y\}$, where $\lozenge\in\{+,\circ,\ast\}$.
Note that if $X$, $Y$ and a function $f$ are definable, so are $X\cup Y$, $X\cap Y$, $X\setminus Y$, $X\times Y$, $f[X]$, $f^{-1}[X]$, $X^{-1}$ and $X\lozenge Y$. Clearly, the additive and multiplicative centres $Z^+(B)$ and $Z^\circ(B)$, the $\lambda$-kernel $\Ker(\lambda)$, the $n^{th}$ socle $\Soc_n(B)$ and the $n^{th}$ annihilator $\Ann_n(B)$ are all $\emptyset$-definable, and for any finite tuple $\bar b\in B$ the additive and multiplicative centralisers $C^+_B(\bar b)$ and $C^\circ_B(\bar b)$ are $\bar b$-definable.

If $\mathcal B$ is a skew brace containing $B$, it is an {\em elementary extension} if every sentence $\varphi$ with parameters in $B$ is true in $B$ iff it is true in $\mathcal B$. If $B$ is infinite, the compactness theorem implies that for every cardinal $\kappa\ge|B|$ there is an elementary extension of cardinal $\kappa$.

A {\em partial type} in variables $\bar x$ over some set of parameters $A\subseteq B$ is a set $\pi(\bar x)$ of $\mathcal{L}(A)$-formulas with free variables $\bar x$ such that any finite subset has a realization in $B$. A partial type $\pi$ is {\em complete} if it is maximal; this is equivalent to saying that for every $\mathcal{L}(A)$-formula $\varphi(\bar x)$ either $\varphi$ or $\neg\varphi$ is in $\pi$. Complete types are just called {\em types}. The compactness theorem states that every partial type is realized in some elementary extension. A skew brace $B$ is {\em $\kappa$-saturated} if every type over $<\kappa$ parameters is already realized in $B$. A skew brace $B$ is {\em saturated} if it is $|B|$-saturated.

We shall call a set $X$ or a function $f$ {\em type-definable} if $X$ or the graph of $f$ are given by an infinite intersection of definable sets. Note that some care is needed when dealing with type-definable sets in insufficiently saturated models, as the intersection may happen to be accidentally too small, or even empty, even though in some elementary extension this is no longer true.

A skew brace $B$ is {\it $\omega$-categorical} if, up to isomorphism, there is only one skew brace of power $\aleph_0$ satisfying $\operatorname{Th}(B)$. The Ryll-Nardzewski Theorem (also proved by Engeler and Svenonius) states that $B$ is $\omega$-categorical iff, for every $n\in\omega$, $\operatorname{Th}(B)$ has only finitely many $n$-types, iff for every $n\in\omega$ there are only finitely many formulas in variables $x_1,\ldots,x_n$ up to equivalence in $\operatorname{Th}(B)$. This in particular implies that for any finite set $A$ of parameters, a set $X$ is $A$-definable iff it is stabilized under all automorphisms fixing $A$ pointwise.

It is easy to see that $\omega$-categoricity is preserved under adding finitely many parameters as constants to the language. Any structure definable in an $\omega$-categorical skew brace (using finitely many parameters) is again $\omega$-categorical. Note that $\omega$-categoricity implies {\it uniform local finiteness}: There is a function $f:\omega\to\omega$ such that any substructure generated by $n$ elements has size at most $f(n)$. In particular, if $B$ is an $\omega$-categorical skew brace, then
\begin{enumerate}[label=$(\arabic*)$]
    \item $(B,+)$, $(B,\circ)$ and $G(B)$ have finite exponent;
    \item every definable skew sub-brace and every quotient $B/I$ for a definable ideal $I$ is still $\omega$-categorical;
\item if $B$ is countably infinite, then a subset of $B^n$ is definable iff it is invariant under all automorphisms of $B$.\end{enumerate}

Let $\lambda$ be an infinite cardinal. Then $B$ is {\it $\lambda$-stable} iff, for every model $A$ of $\operatorname{Th}(B)$ of cardinality $\lambda$, the set of complete types over $A$ has cardinality $\lambda$. A skew brace $B$ is {\it stable} if it is $\lambda$-stable for some infinite cardinal $\lambda$; it is {\em superstable} if it is $\lambda$-stable for all $\lambda\ge 2^{\aleph_0}$. It turns out that $\omega$-stability implies superstability, which obviously implies stability. We refer the reader to \cite{poizat} and \cite{wa97} for results and properties of stable and $\omega$-stable groups. Note that $\lambda$-stability is preserved under adding $\le\lambda$ parameters to the language. Any structure definable in a $\lambda$-stable skew brace is again $\lambda$-stable.

A theory has the strict order property SOP if there is a formula $\varphi(\bar x,\bar y)$ which defines a partial pre-order with infinite chains on some definable set; otherwise the theory is NSOP. Any stable theory is NSOP, and NSOP is preserved under adding parameters to the language; any structure definable in an NSOP skew brace is again NSOP.

\section{Stable Skew Braces}\label{stability}
\subsection{Chain conditions, connected components and generic types}
We shall need the following facts, which can be found in \cite{poizat} and \cite{wa97}. We shall formulate them for type-definable groups, i.e.\ groups given by some partial type. It can be shown that in a stable theory, type-definable groups can actually be given as an intersection of definable supergroups. When dealing with type-definable groups, some care has to be taken, as we want all the relevant properties to be preserved in elementary extensions. This is not a problem if the model is sufficiently saturated; if not, in the given model there may be too few elements realising the infinite intersection given by the partial type.

We also may want to consider {\em arbitrary} (neither definable nor undefinable) subgroups of a stable group; we consider them just as a set of parameters, which does not change when going to an elementary extension. 

\begin{fact}\label{f:icc} In a stable group, for any family $(H_i:i\in\mathcal{I})$ of uniformly definable subgroups, there is an integer $n$ depending only on the formula defining the $H_i$ and $\mathcal{J}\subseteq\mathcal{I}$ of size at most $n$ such that $\bigcap_{i\in\mathcal I}H_i=\bigcap_{i\in\mathcal J}H_i$.
\end{fact}

This allows us to define various subgroups related to a subgroup $H$ of a stable group $G$. Let $\varphi(x,\bar y)$ a formula.\begin{itemize}
\item If $H$ is definable or arbitrary, we put $\mathfrak X^0_\varphi(H)$ to be the collection of all $\varphi$-definable subgroups of $G$ intersecting $H$ in a subgroup of finite index.
\item If $H=\bigcap_{i\in I} H_i$ is type-definable, we put $\mathfrak X^0_\varphi(H)$ to be the collection of all $\varphi$-definable subgroups of $G$ intersecting $\bigcap_{i\in I_0}H_i$ in a subgroup of finite index for some finite $I_0\subseteq I$.
\item If $H$ is arbitrary, we put $\mathfrak X_\varphi(H)$ to be the collection of all $\varphi$-definable subgroups of $G$ containing $H$.\end{itemize}
Then $H^0_\varphi=\bigcap\mathfrak X^0_\varphi(H)$ is a definable subgroup of $G$, over the same parameters as $H$ if $H$ is definable or type-definable, and over $H$ if $H$ is arbitrary, and intersecting $H$ (or some $\bigcap_{i\in I_0}H_i$ for $I_0\subseteq I$ finite if $H$ is type-definable) in a subgroup of finite index. It is in fact the intersection of all such $\varphi$-definable subgroups.

The {\em connected component} of $H$ is defined as $H^0=\bigcap_\varphi H^0_\varphi$ for definable or type-definable $H$, and as $H^0=H\cap \bigcap_\varphi H^0_\varphi$ for arbitrary $H$.

For arbitrary $H$, the intersection $\bar H_\varphi=\bigcap\mathfrak X_\varphi(H)$ is an $H$-definable subgroup of $G$ containing $H$, and $\bar H=\bigcap_\varphi \bar H_\varphi$ is the {\em definable hull} of $H$. It is the intersection of all definable subgroups containing $H$. We have $\mathfrak X^0_\varphi(H)=\mathfrak X^0_\varphi(\bar H)$, so $H^0_\varphi=(\bar H)^0_\varphi\cap H$ and $H^0=(\bar H)^0\cap H$.

We call $H$ {\em connected} if $H=H^0$. Note that in the stable case the connected component even of a definable group is only type-definable. It is easy to see that a connected component is connected and {\em definably characteristic}, i.e.\ stabilised by all definable automorphisms. Moreover, as every $G^0_\varphi$ has finite index, $|G:G^0|\le 2^{\aleph_0}$; if $H^0$ is definable, it has finite index in $H$.

If $A$ is an additive/multiplicative subgroup of a skew brace, we shall denote the additive/multiplicative definable hull
by $\bar A^+$ and $\bar A^\circ$, respectively.

\begin{rem}So far we have only considered type-definable subgroups given as an intersection of definable subgroups. In a stable theory this is a general phenomenon: every group given by some partial type is in fact also given as an intersection of definable subgroups.\end{rem}

\begin{rem}\label{r:ideal} 
Given an additive subgroup $A$ and a multiplicative subgroup $G$ of a skew brace $B$, we put $\Stab_G(A)=\{g\in G:\lambda_g(A)= A\}$, the {\em stabilizer} of $A$ in $G$. If $G=B$ it is omitted. Note that $g\in G$ is in $\Stab_G(A)$ iff $g\circ A= g+A$. Clearly, $\Stab_G(A)$ is a multiplicative subgroup of $G$; if both $G$ and $A$ are definable, so is $\Stab_G(A)$.

If $H$ is a definable additive subgroup containing $A$, or intersecting $A$ in a subgroup of finite index, we can consider $H^\lambda=\bigcap_{b\in\Stab(A)}\lambda_b(H)$. This is again a definable additive subgroup containing $A$, or intersecting $A$ in a subgroup of finite index, and $\Stab(H^\lambda)\ge\Stab(H)$. In particular, if $A$ is a left ideal, so is $H^\lambda$, and $\bar A^+$ and $(\bar H^+)^0$ are type-definable as intersections of left ideals.

However, if $A$ is merely a skew sub-brace, it is difficult to obtain a small definable skew sub-brace containing $A$ (e.g.\ contained in $H$); if $A$ is an ideal, it is difficult to obtain a small definable ideal containing $A$.\end{rem}

\begin{fact} If $G$ is $\omega$-stable, it has the descending chain condition on definable subgroups. In particular, $G^0$ is definable and has finite index in $G$. If $G$ is stable $\omega$-categorical, since all $G^0_\varphi$ are definable over the same parameters as $G$, there are only finitely many possibilities for $G^0_\varphi$ and $G^0$ again has finite index. Finally, if $G$ is $\omega$-stable of finite Morley rank, it also has the ascending chain condition on connected definable subgroups.\end{fact}

\begin{fact} Let $G$ be a connected stable group. Then there is a unique type $p$ over $G$ such that $g\circ p=p\circ g=p^{-1}=p$ for all $g\in G$, where $g\circ p=\tp(g\circ a/G)$, $p\circ g=\tp(a\circ g/G)$ and $p^{-1}=\tp(a^{-1}/G)$ for some/any realization $a$ of $p$. Conversely, $G$ is connected if there is a type $p$ such that $g\circ p=p$ for all $g\in G$.\end{fact}
The type $p$ is called {\it generic} for $G$. It is fixed under all definable automorphisms of~$G$. For a general type-definable stable group, the {\em principal} generic type is the generic type of the connected component $G^0$; a type is generic if it is a translate of the principal generic type.

\begin{fact}[{\cite{wa90}}]\label{f:finsat} Let $G$ be a stable group and $H$ an arbitrary subgroup of $G$. Then the generic types of $\bar H$ are finitely satisfiable in $H$, i.e.\ any finite subset $\pi$ has a realization in $H$. Moreover, any generic type of $\bar H{}^0$ is finitely satisfiable in $H\cap H_0$, for any definable supergroup $H_0\ge \bar H{}^0$.\end{fact}

Note that we cannot expect a generic type of $\bar H{}^0$ to be finitely satisfiable in $H^0$, as $H^0$ might be trivial.

\begin{theo}\label{conncpt} Let $B$ be a stable skew brace, and $B^0$ the additive connected component. Then $B^0$ is multiplicatively connected, and an ideal of index $\le 2^{\aleph_0}$ in $B$; if $B$ is $\omega$-stable or $\omega$-categorical, $B^0$ has finite index.\end{theo}
\begin{proof} The maps $\lambda_a$ for $a\in B$ are definable additive automorphisms, and must stabilise $B^0$. Thus $B^0$ is a strong left ideal, in particular a multiplicative subgroup. Now for any $a\in B^0$ and $b$ realising the principal generic type over $B$ we have $p=\lambda_a(p)=\tp(-a+a\circ b/B)$, so $p=a+p=\tp(a\circ b/B)=a\circ p$. It follows that $p$ is also the unique multiplicative generic type of $B^0$, so $B^0$ is multiplicatively connected, and must be equal to the multiplicative connected component. Hence $B^0$ is both additively and multiplicatively normal in $B$, i.e.\ an ideal. The statement about the index follows from the same property for groups.\end{proof}

\begin{theo} Let $B$ be a stable skew  brace, and $C$ an arbitrary skew sub-brace. Then $(\bar C^+)^0\cap C=(\bar C^\circ)^0\cap C$ is a connected ideal of index  $\le 2^{\aleph_0}$ in $C$; if $B$ is $\omega$-stable or $\omega$-categorical, the index is finite.\end{theo}
\begin{proof} By Remark \ref{r:ideal} the connected component $(\bar C^+)^0$ is an intersection of $C$-definable additive subgroups $(H_i:i\in I)$ stabilised and additively normalized by $C$. In particular $H_i\cap C$ is a strong left ideal of finite index in $C$. Recall that for left ideals additive and multiplicative left cosets coincide, so additive and multiplicative index in $C$ are the same. Moreover, for $c\in H_i\cap C$ we have $c\circ H_i=c+\lambda_c(H_i)\subseteq H_i$. By stability we must have equality, so $H_i\cap C\le M_i:=\{b\in B:c\circ H_i=H_i\}$, the multiplicative stabiliser of $H_i$. Then 
$$(C^+)^0=(\bar C^+)^0\cap C=\bigcap_{i\in I}H_i\cap C\le\bigcap_{i\in I}M_i.$$
Now $(\bar C^+)^0$ has a unique (additive) generic type $p$ which is finitely satisfiable in every $H_i\cap C$ by Fact \ref{f:finsat}. Moreover, every $c\in\bigcap_{i\in I}M_i$ stabilises $\bigcap_{i\in I}H_i$ multiplicatively, whence the unique generic type $p$, and $p=c\circ p$. It follows that $\bigcap_{i\in I}M_i$ is connected.

Suppose that $(C^+)^0$ is strictly contained in $\bigcap_{i\in I}M_i\cap C$. Then there is $c_0\in(\bigcap_{i\in I}M_i\cap C)\setminus (C^+)^0$. Choose $i\in I$ such that $c_0\notin H_i$. Since $x\in H_i$ is in $p$, the formula $c_0^{-1}\circ x\in H_i$ is in $c_0\circ p=p$. But $H_i\cap c_0\circ H_i\cap C$ is empty (recall that $H_i\cap C$ is a multiplicative group and $c_0\in C$), contradicting finite satisfiability. 

Suppose $\bar C^0<\bigcap_{i\in I}M_i$. As $\bigcap_{i\in I}M_i$ is connected, there is a definable multiplicative subgroup $G$ intersecting $\bigcap_{i\in I}M_i$ in a subgroup of infinite index, but $C$ in a subgroup of finite index. Replacing $G$ by $\bigcap_{c\in C}c^{-1}\circ G\circ c$ we may assume $C\le N_B^\circ(G)$. Then $(C^+)^0\circ G$ is a definable finite extension of $G$ containing all of $(C^+)^0$, but still intersecting $\bigcap_{i\in I}M_i$ in a subgroup of infinite index. So $x\notin (C^+)^0\circ G$ is a generic formula for $\bigcap_{i\in I}M_i$ and must be in $p$, contradicting finite satisfiability of $p$ in $C$.

Thus $\bigcap_{i\in I}M_i\cap C=(C^+)^0$, and $C^0=(C^+)^0=(C^\circ)^0=(\bar C^+)^0\cap C=(\bar C^\circ)^0\cap C$. By the additive characterisation it is a strong left ideal, and by the multiplicative characterisation it is multiplicatively normal, whence an ideal in $C$, of index $\le 2^{\aleph_0}$.

If $B$ is $\omega$-stable or $\omega$-categorical, the connected components are definable and the indices are finite.\end{proof}

We put $C^0=(\bar C^+)^0\cap C=(\bar C^\circ)^0\cap C$, the {\em connected component} of $C$. Note that if $(\bar C^+)^0$ has infinite index in $\bar C^+$, then $C^0$ may be trivial.

\subsection{Internality and $\Phi$-components}
Let $\Phi$ be an invariant collection of partial types. A partial type $\pi$ is {\em $\Phi$-internal} if there are finitely many partial types $\phi_1,\ldots,\phi_n$ in $\Phi$ and a definable function $f$ with domain $\phi_1\times\cdots\times\phi_n$ whose image contains $\pi$. We may have to add new parameters to define $f$.

An {\em imaginary (element)} is the class of some tuple modulo an $\emptyset$-definable equivalence relation. For instance, a tuple itself can be viewed as an imaginary element. If $\bar a$ is an imaginary, then the {\em definable closure} $\dcl(\bar a)$ is the tuple of all imaginary elements $e$ such that there is a formula $\varphi(x,\bar a)\in\mathcal L(\bar a)$ whose only realisation is $e$. We say that $e$ is {\em definable} over $\bar a$.

\begin{fact}
Let $G$ be a stable group, $g$ realize its principal generic type, and put
$$g_\Phi=(g_0\in\dcl(g):\tp(g_0)\mbox{ is $\Phi$-internal}).$$ Then there is a unique normal type-definable subgroup $N$ of $G$ such that $\dcl(g_\Phi)=\dcl(gN)$. It does not depend on the choice of $g$, and is definably characteristic and automorphism-invariant.\end{fact}

In fact, $N$ is the intersection of all definable subgroups $H$ such that $G/H$ is $\Phi$-internal. Note that a priori $N$ need not be connected. We call $N$ the {\em $\Phi$-component} of $G$, denoted $G^\Phi$. Recall that if $G$ is $\omega$-stable, $G^\Phi$ is outright definable.

\begin{theo} Let $B$ be a connected stable skew brace, and $\Phi$ an invariant collection of partial types. Then the additive $\Phi$-component $B^\Phi_+$ equals the multiplicative $\Phi$-component $B^\Phi_\circ$ and is just called the $\Phi$-component $B^\Phi$; it is an ideal in $B$.\end{theo}
\begin{proof} Since $B^\Phi_+$ is definably characteristic, it is additively normal and invariant under $\lambda_b$ for all $b\in B$, whence a strong left ideal.

Thus $B^\Phi_+$ is a multiplicative subgroup, and the multiplicative quotient $B/B^\Phi_+$ is $\Phi$-internal (since additive and multiplicative cosets are the same). But then $B^\Phi_\circ\le B^\Phi_+$, and for a generic $g$ we have
$$\dcl(gB^\Phi_+)=\dcl(g_\Phi)=\dcl(gB^\Phi_\circ).$$

Suppose that $B^\phi_\circ<B^\Phi_+$, let $h\in B^\Phi_+\setminus B^\Phi_\circ$, and let $g$ be generic over $h$. Then $gh$ is also generic; as $B$ is connected, there is a unique generic type, and there is an automorphism $\sigma$ of $B$ mapping $g$ to $gh$. So $\sigma$ fixes $gB^\Phi_+$ and moves $gB^\Phi_\circ$ to $ghB^\phi_\circ\not= gB^\Phi_\circ$, contradicting $gB^\Phi_\circ\in\dcl(gB^\phi_+)$. Hence $B^\Phi_+=B^\Phi_\circ=:B^\Phi$. It follows that $B^\Phi$ is multiplicatively normal, whence an ideal.
\end{proof} 

Note that if $I$ is a connected type-definable ideal of $B$, then $I^\phi$ is additively and multiplicatively definably characteristic, whence an ideal in $B$.

In a superstable theory there is an ordinal-valued rank $U$ on the collection of all types, called {\em Lascar rank}, such that $U(p)=0$ iff $p$ is {\em algebraic}, i.e.\ has only finitely many realisations (in any elementary extension). It satisfies the {\em Lascar inequalities}
$$U(\bar a/A\bar b)+U(\bar b/A)\le U(\bar a\bar b/A)\le U(\bar a/A\bar b)\oplus U(\bar b/A),$$
where $\oplus$ is the natural sum of ordinals: If $\alpha=\sum_{i\le k}\omega^{\gamma_i}\cdot m_i$ and $\beta=\sum_{i\le k}\omega^{\gamma_i}\cdot n_i$ for some $\gamma_0>\gamma_1>\cdots>\gamma_k$ and $m_i,n_i\in\omega$, then $\alpha\oplus\beta=\sum_{i\le k}\omega^{\gamma_i}\cdot(m_i+n_i)$.

In the context of groups $H\le G$ the Lascar inequalities translate as
$$U(H)+U(G/H)\le U(G)\le U(H)\oplus U(G/H),$$
where the rank of a type-definable group $G$ is the rank of any of its generic types (and the generic types of $G$ are precisely the types in $G$ of maximal rank).

In an $\omega$-stable theory there is another ordinal-valued rank, {\em Morley rank} $RM$, which is also defined on formulas. We have $RM(p)\ge U(p)$ for any type $p$. In groups of finite Morley rank, the two ranks coincide.

\begin{cor}[Berline-Lascar analysis]\label{c:components}
If $U(B)=\sum_{i\le k}\omega^{\alpha_i}\cdot n_i$ and $\Phi=\{p:U(p)<\omega^{\alpha_j}\}$, then $(B^0)^\Phi$ is the unique type-definable connected ideal of Lascar-rank $\sum_{i\le j}\omega^{\alpha_i}\cdot n_i$.
\end{cor}

We shall need the following {\em Indecomposability Theorem} (which we state in the version without indecomposability \cite[Theorem 5.4.5]{wa00}).
\begin{fact}\label{f:indecomposability} Let $G$ be a superstable group with $U(G)<\omega^{\alpha+1}$ and $\mathfrak X$ a collection of non-empty type-definable subsets of $G$ closed under inverse, union and product. Then there is a type-definable connected subgroup $H\subseteq X$ for some $X\in\mathfrak X$ with $U(H)=\omega^\alpha\cdot k$ for some $k\in\omega$ such that $U(X/H)<\omega^\alpha$ for all $X\in\mathfrak X$.\end{fact}

We shall call a skew brace {\em definably simple} if it has no definable proper non-zero ideal, and {\em type-definably simple} if it has no type-definable proper non-zero ideal. A superstable type-definably simple group is simple; the next theorem shows the analogous result for skew braces.

\begin{theo} A superstable type-definably simple skew brace is simple. An $\omega$-stable definably simple skew brace is simple.\end{theo}
\begin{proof}The second statement follows from the first, as in an $\omega$-stable skew brace a type-definable ideal is definable. 
	
Let $B$ be a type-definably simple skew brace. If $B$ is trivial, the result follows from the analogous theorem for groups. So suppose $B$ is non-trivial, and $I\triangleleft B$ is a proper non-zero ideal. Note first that $U(B)=\omega^\alpha\cdot n$ for some ordinal $\alpha$ and some $n<\omega$ by Corollary \ref{c:components}. 

As $\Soc(B)$ is a definable proper ideal, it must be zero. Choose $0\not=a\in I$, then $a\notin\Soc(B)$, so either $C_+(a)<B$ or $\Fixl(a)=\{b\in B:\lambda_b(a)=a\}<B$; note that the latter is a multiplicative subgroup. In the first case put $X=\{b+a-b:b\in B\}$, in the second case put $X=\{\lambda_b(a):b\in B\}$. These are in bijection with $B^+/C_+(a)$ and $B^\circ/\Fixl(a)$ respectively; as $U(C_+(A))<\omega^\alpha\cdot n$ or $U(\Fixl(a))<\omega^\alpha\cdot n$, the Lascar inequalities yield that the rank of the quotient is at least $\omega^\alpha$, so $U(X)\ge\omega^\alpha$. Note that $X\subseteq I$.

Let $\mathfrak X$ be the smallest collection of definable subsets of $B$ containing $X$ and such that if $Y,Z\in\mathfrak X$ and $b\in B$, then
$$Y\cup Z,\ Y-Z,\ Y\circ Z^{-1},\ b+Y-b,\ b\circ Y\circ b^{-1}, \mbox{ and }\lambda_b(Y)$$
are in $\mathfrak X$. So $\mathfrak X$ consists of definable subsets of $I$, and $\bigcup\mathfrak X$ is the ideal generated by $X$. By Fact \ref{f:indecomposability} there is a connected type-definable additive subgroup $H\subseteq Y$ for some $Y\in\mathfrak X$ such that $U(H)=\omega^\alpha\cdot k$ and $U(Z/H)<\omega^\alpha$ for all $Z\in\mathfrak X$. It is unique, since if $K\not=H$ were a second such group, then $U(H\cap K)<\omega^\alpha\cdot k$, so $U(K+H/H)\ge\omega^\alpha$; a contradiction since $K+H$ is contained in some set in $\mathfrak X$. In particular $H$ is additively normal and $\lambda$-invariant, whence a strong left ideal, and a connected multiplicative subgroup. If there were $b\in B$ with $b\circ H\circ b^{-1}\not= H$ then $U(H\circ b\circ H\circ b^{-1}/H)\ge\omega^\alpha$, again a contradiction. Thus $H$ is a non-zero type-definable ideal contained in $I$, contradicting type-definable simplicity.
\end{proof}

\section{Solubility}\label{solubility}
There are many possible definitions of solubility in skew braces. The weakest one was introduced in~\cite{konovalov}. Let~$B$ be a skew brace. We recursively define $B_{(i)}$ as follows: 
$$B_{(1)}=B\quad\mbox{and}\quad B_{(n+1)}=\langle B_{(n)}\ast B_{(n)}\rangle_+.$$
Note that $B_{(n+1)}$ is an ideal in $B^{(n)}$ for all $n<\omega$. Then $B$ is {\em weakly soluble} of {\em derived length} $k$ if $B_{(k)}>B_{(k+1)}=\{0\}$. This means that there is a sequence $$\{0\}=I_0<I_1<\cdots<I_k=B$$ of skew sub-braces such that $I_n$ is an ideal in $I_{n+1}$ with trivial quotient for all $n<k$.

We shall say that $B$ is {\em left soluble} of derived length $\le k$ if in the above sequence all $I_n$ are left ideals in $B$; it is {\em strongly left soluble} of derived length $\le k$ if all the $I_n$ are strong left ideals in $B$, and it is {\em soluble} if all $I_n$ are ideals in $B$. Note that if $B$ is of soluble type, then we can require that all the quotients $I_{n+1}/I_n$ are abelian.

The weak soluble/left soluble/strongly left soluble/soluble radical of $B$ is the ideal generated by all weak soluble/left soluble/strongly left soluble/soluble ideals. Usually it need not be weak soluble/left soluble/strongly left soluble/soluble itself.

We shall now introduce some terminology.
Recall that if $H$ is a subgroup of $G$ and $x\in N_G(H)$, then $C_G(x/H)=\{g\in N_G(H):[g,x]\in H\}$. Clearly, this is a subgroup of $G$ containing $H$; if both $G$ and $H$ are definable, so is $C_G(x/H)$. For a subset $X\subseteq N_G(H)$ we put $C_G(X/H)=\bigcap_{x\in X}C_G(x/H)$, again a subgroup. This will be used both additively and multiplicatively.

Following \cite{FCbraces}, given a multiplicative subgroup $G$ and an additive subgroup $H$ of $B$ and an element $c\in B$, we define 
$$\Fixl_G(c/H)=\{x\in\Stab_G(H)\,:\,\lambda_x(c)\in c+H\}.$$
For a set $C\subseteq B$ we put $\Fixl_G(C/H)=\bigcap_{c\in C}\Fixl_G(c/H)$. Since for $x,x'\in\Fixl_G(c/H)$ we have
$$\lambda_{x\circ x'}(c)=\lambda_x(\lambda_{x'}(c))\in\lambda_x(c+H)=\lambda_x(c)+\lambda_x(H)=c+H,$$
both $\Fixl_G(c/H)$ and $\Fixl_G(C/H)$ are multiplicative subgroups of $G$ (for the inverse, consider $x=x'^{-1}$). Clearly, if $G$ and $H$ are definable, so is $\Fixl_G(c/H)$. If $H$ is trivial, we just write $\Fixl_G(c)$.

Finally, if $H$ and $G$ are additive subgroups and $c\in B
$, we define
$$\Fixr_G(c/H)=\{x\in N^+_G(H)\,:\,\lambda_c(x)\in x+H\},$$
and for a set $C$ we put $\Fixr_G(C/H)=\bigcap_{c\in C}\Fixr_G(c/H)$. Since for $x,x'\in\Fixr(c/H)$ we have
$$\lambda_c(x+x')=\lambda_c(x)+\lambda_c(x')\in x+H+x'+H=x+x'+H,$$
both $\Fixr_G(c/H)$ and $\Fixr_G(C/H)$ are additive subgroups of $G$ (for the inverse, consider $x=-x'$). Clearly, if $G$ and $H$ are definable, so is $\Fixr_G(c/H)$. As above we write $\Fixr_G(c)$ if $H$ is trivial.

\begin{rem}For $x\in\Stab_B(H)$ and $y\in N^+_B(H)$ we have $x\in\Fixl_B(y/H)$ iff $y\in\Fixr_B(x/H)$ iff $\lambda_x(y)\in y+H$.\end{rem}
\begin{rem}\label{r:contained} For a skew sub-brace $A$ and $X\subseteq\Stab_B(A)$ we have $A\le\Fixr_B(X/A)$. However, for $X\subseteq\Stab_B(A)\cap N_B^+(A)$ we have $A\le\Fixl_B(X/A)$ iff $X\le N_B^\circ(A)$.\end{rem}
\begin{proof} The first assertion is clear by definition. For the second assertion, consider $a\in A$ and $x\in X$. Then
$$\lambda_a(x)=-a+a\circ x\in x+A=A+x\Leftrightarrow a\circ x\in A+x=x+A=x\circ A\Leftrightarrow x^{-1}\circ a\circ x\in A.\qedhere$$\end{proof} 

\begin{rem}\label{fixdef}
Let $B$ be a stable skew brace, and $A$ (resp.\ $G$) be a definable additive (resp.\ multiplicative) subgroup. If $X$ is any subset of $B$, then there is a finite subset $Y$ of $X$ such that $$\begin{aligned}\Fixl_B(X/A)&=\Fixl_B(Y/A),&\Fixr_B(X/A)&=\Fixr_B(Y/A),\\C_B^+(X/A)&=C_B^+(Y/A),\quad\text{and}& C_B^\circ(X/G)&=C_B^\circ(Y/G),\end{aligned}$$
where $C^+$ is the additive and $C^\circ$ the multiplicative centraliser. In particular, all of these groups are definable.
\end{rem}
\begin{proof}This is an immediate application  of Fact \ref{f:icc}.\end{proof}

The following two Lemmas will allow us to find definable skew sub-braces.
\begin{lem}\label{l:norm} Let $B$ be a skew brace, $A$ a skew sub-brace, and $c\in N^+_B(A)$ such that $A\ast c\subseteq A$. Then $c\in N^\circ_B(A)$ iff $c\in\Stab_B(A)$.\end{lem}
\begin{proof}Suppose first that $c\in N^\circ_B(A)$. Then
$$\lambda_c(A)=-c+c\circ A=-c+A\circ c=-c+A+c+A=A.$$
Next, assume $c\in\Stab_B(A)$, and consider $a\in A$. Then $-a+a\circ c-c=a\ast c\in A$, whence 
	$a\circ c\in A+c=c+A=c\circ A$ and $c^{-1}\circ a\circ c\in A$.\end{proof}

\begin{lem}\label{l:ast} Let $B$ be a skew brace, $H$ an additive subgroup, $G$ a multiplicative subgroup and $A\subseteq G\cap H$ a skew sub-brace such that\begin{itemize}
\item $H\le N^+_B(A)$, and $G\le\Stab_B(A)$.
\item $G\ast H\subseteq A$ or $H\ast G\subseteq A$.
\end{itemize}
Then $C:=G\cap H$ is a skew sub-brace containing $A$ as an ideal with trivial quotient $C/A$.\end{lem}
\begin{proof}Consider $b,b'\in C$. Then $b\ast b'\in A$, so $b+b'\in b\circ b'+A$; as $b\circ b'\in\Stab_B(A)$ we have $b\circ b'+A=b\circ b'\circ A$. This implies $b+b'\in G$ and $b\circ b'\in H$, whence $b+b'\in C$ and $b\circ b'\in C$. 

To see that $C$ is closed under inverse, consider $b\in C$ and note that
$b^{-1}\in G$. Hence either $b^{-1}\ast b\in A$ or $b\ast b^{-1}\in A$, which means $b^{-1}+b\in A$ or $b+b^{-1}\in A$. Since $b$ normalizes $A$, in either case $b^{-1}\in -b+A$, whence $-b\in b^{-1}+A=b^{-1}\circ A$ as $b^{-1}\in\Stab_B(A)$. Thus $b^{-1}\in H$ and $-b\in G$, implying $b^{-1}\in C$ and $-b\in C$. Hence $C$ is a skew sub-brace containing $A$. Moreover $A$ is additively normal in $C\le H$; since $C\ast C\subseteq A$, firstly $A$ is an ideal in $C$ and secondly the quotient $C/A$ is trivial.\end{proof}

We shall now study type-definable hulls.
\begin{lem}\label{l:hulls} Let $B$ be a stable brace, $X$ a subset, $H$ an additive and $G$ a multiplicative subgroup of $B$. Then $$\begin{aligned}\bar{N_B^+(H)}^+&\le N_B^+(\bar H^+),&
\bar{N_B^\circ(G)}^\circ&\le N_B^\circ(\bar G^\circ),& \bar{\Stab_B(H)}^\circ&\le\Stab_B(\bar H^+),\\ \bar{\Fixr_B(X/H)}^+&\le\Fixr(X/\bar H^+),&&\mbox{and}& \bar{\Fixl_B(X/H)}^\circ&\le\Fixl(X/\bar H^+).\end{aligned}$$\end{lem}
\begin{proof} Let $N$ be a definable supergroup of $H$. Then
	$$\bigcap_{n\in N^+_B(H)}n+N-n$$
is a definable subgroup of $N$ containing $H$ and additively normalized by 
$N_B^+(H)$. Hence $\bar{N_B^+(H)}^+\le N_B^+(\bar H^+)$. The proof for the multiplicative normalizer is analogous. 

For the stabilizer, consider again a definable additive $N\ge H$. Then
$$\bigcap_{s\in\Stab_B(H)}\lambda_s(N)$$
is a definable subgroup of $N$ containing $H$ and stabilized by $\Stab_B(H)$. Hence $\bar{\Stab_B(H)}^\circ\le\Stab_B(\bar H^+)$.

Now $\bar H^+$ is an intersection of definable supergroups $N$ normalized by $N_B(H)$. Then
$\Fixr_B(X/H)\le\Fixr(X/N)$, yielding $\bar{\Fixr_B(X/H)}^+\le\Fixr(X/\bar H^+)$. Similarly, $\bar H^+$ is an intersection of definable supergroups $N$ stabilized by $\Stab_B(H)$, whence
$\Fixl_B(X/H)\le\Fixl(X/N)$, and $\bar{\Fixl_B(X/H)}^\circ\le\Fixl(X/\bar H^+)$.\end{proof}

\begin{theo}\label{t:wsol}Let $B$ be a stable skew brace, and $C$ a weakly soluble skew sub-brace. Then $\bar C^+=\bar C^\circ=\bar C$ is a type-definable weakly soluble skew sub-brace $S$ of the same derived length, and $\bar{C_{(n)}}\ast\bar{C_{(n)}}\subseteq\bar{C_{(n+1)}}$ for all $n<\omega$. If $B$ is $\omega$-stable, $\bar C$ and all $\bar{C_{(n)}}$ are definable.\end{theo}
\begin{proof} The last assertion is clear, since in an $\omega$-stable group type-definable subgroups are definable.
	
Suppose the derived length of $C$ is $k$. We shall show by reverse induction on $n$ that the additive definable hull $\bar{C_{(n)}}^+$ and the multiplicative definable hull $\bar{C_{(n)}}^\circ$ coincide for all $n$, and that  $\bar{C_{(n)}}\ast\bar{C_{(n)}}\subseteq\bar{C_{(n+1)}}$. Clearly we may assume that the ambient brace $B$ is $\aleph_1$-saturated. 

There is nothing to show for $n=k+1$. So suppose that $\bar{C_{(n+1)}}^+=\bar{C_{(n+1)}}^\circ=\bar{C_{(n+1)}}$ is a type-definable skew sub-brace. By Lemma \ref{l:hulls} we have
$$\begin{aligned}\bar{C_{(n)}}^+&\le\bar{N_B^+(C_{(n+1)})}^+\le N_B^+(\bar{C_{(n+1)}}),\\
\bar{C_{(n)}}^\circ&\le\bar{N_B^\circ(C_{(n+1)})}^\circ\le N_B^\circ(\bar{C_{(n+1)}}),\quad\mbox{and}\\
\bar{C_{(n)}}^\circ&\le\bar{\Stab_B(C_{(n+1)})}^\circ\le\Stab_B(\bar{C_{(n+1)}}).\end{aligned}$$
Moreover,
$$\bar{C_{(n)}}^+\le\bar{\Fixr_B(C_{(n)}/C_{(n+1)})}^+\le\Fixr_B(C_{(n)}/\bar{C_{(n+1)}}),$$
so
$$\bar{C_{(n)}}^\circ\le\bar{\Fixl_B(\bar{C_{(n)}}^+/\bar{C_{(n+1)}})}^\circ\le
\Fixl_B(\bar{C_{(n)}}^+/\bar{C_{(n+1)}}).$$

By Lemma \ref{l:ast} the intersection $\bar{C_{(n)}}^\circ\cap\bar{C_{(n)}}^+$ is a skew sub-brace containing $\bar {C_{(n+1)}}$ as an ideal with trivial quotient, and it is clearly type-definable (this uses $\aleph_1$-saturation, as we need it to type-define a skew sub-brace even in an elementary extension). Now by stability a type-definable subgroup is an intersection of definable subgroups. Applying this additively and multiplicatively, we obtain $\bar{C_{(n)}}^\circ=\bar{C_{(n)}}^+=\bar{C_{(n)}}$.
	
This finishes the induction step. In particular $\bar{C_{(1)}}^+=\bar{C_{(1)}}^\circ=\bar C$ is a type-definable weakly soluble skew sub-brace of derived length $k$, as required.\end{proof}

\begin{cor}Let $B$ be a stable skew brace, and $C$ a left soluble/strongly left soluble/soluble skew sub-brace, and let $\{0\}=I_0<I_1<\cdots<I_k=C$ be the corresponding chain of left ideals/strong left ideals/ideals. Then $C$ is contained in a type-definable skew sub-brace $\bar C$ of the same solubility type, and $\{0\}=\bar I_0<\bar I_1<\cdots<\bar I_k=\bar C$ is the corresponding chain of left ideals/strong left ideals/ideals.\end{cor}
\begin{proof}Since all $I_n$ are weakly soluble skew sub-braces, $\bar I_n^+=\bar I_n^\circ=\bar I_n$ is well-defined. The rest follows from Lemma \ref{l:hulls}.\end{proof}

\section{Nilpotency}\label{nilpotency}
\subsection{Annihilator-nilpotency}
With the terminology from the previous section, we slightly generalize the definition of the annihilator. Let $I\trianglelefteq B$ be an ideal. We put
$$\begin{aligned}\Ann(B;I)&=\Ann_1(B;I)=\Fixl_B(B/I)\cap C_B^+(B/I)\cap C_B^\circ(B/I)\quad\mbox{and}\\\
Ann_{n+1}(B;I)&=\Ann(B;\Ann_n(B;I))\\
&=\Fixl_B(B/\Ann_n(B;I))\cap C^+_B(B/\Ann_n(B;I))\cap C^\circ_B(B/\Ann_n(B;I)).\end{aligned}$$
We put $\overline{\Ann}(C;I)=\bigcup_{n<\omega}\Ann_n(C;I)$. If $I=\{0\}$ it is omitted. Note that because of centrality we also have
$$Ann_{n+1}(B;I)=\Fixr_B(B/\Ann_n(B;I))\cap C^+_B(B/\Ann_n(B;I))\cap C^\circ_B(B/\Ann_n(B;I)).$$
We can now define the lower annihilator series by
$$\Gamma_0(B)=B,\quad\mbox{and}\quad\Gamma_{n+1}(B)=\langle\Gamma_n(B)\ast B,\,[B,\Gamma_n(B)]_\circ,\,[B,\Gamma_n(B)]_+\rangle_+.$$
These are all ideals in $B$ and $B/I$ is annihilator nilpotent of class $\le k$ iff $\Gamma_k(B)\le I$; in this case $\Gamma_n(B)\le\Ann_{k-n}(B/I)$ for all $0\le n\le k$.

\begin{lem}\label{l:locbddannnilp}Let $B$ be a skew brace which is locally annihilator nilpotent of class $\le k$. Then $B$ is annihilator nilpotent of class $\le k$.\end{lem}
\begin{proof} We have
$\Gamma_k(B)=\bigcup_{C\le B\text{ fin.\ gen.}}\Gamma_k(C)=\{0\}$.
\end{proof}

\begin{lem}\label{anncdiversozero}
Let $B$ be a stable skew brace, $N$ a definable skew sub-brace, and $C$ a skew sub-brace stabilising and normalizing (additively and multiplicatively) $N$. If $C/(C\cap N)$ is non-zero locally annihilator nilpotent, then $\Ann(C;C\cap N)\not\le N$.
\end{lem}
\begin{proof}
It follows from Fact \ref{f:icc} that there is a finite set $X\subseteq C$ such that $$\Fixl_B(X/N)=\Fixl_B(C/N),\quad C^+_B(X/N)=C^+_B(C/N),\quad\textnormal{and}\quad C^\circ_B(X/N)=C^\circ_B(C/N).$$
Let $C_0$ be the skew sub-brace generated by $X$. Then
$$\begin{aligned}\Ann(C;C\cap N)&=\Fixl_C(C/(C\cap N))\cap C^+_C(C/(C\cap N))\cap C^\circ_C(C/(C\cap N))\\
&= C\cap\Fixl_B(C/N)\cap C^+_B(C/N)\cap C^\circ_B(C/N)\\
&=C\cap\Fixl_B(X/N)\cap C^+_B(X/N)\cap C^\circ_B(X/N)\\
&=C\cap\Fixl_B(C_0/N)\cap C^+_B(C_0/N)\cap C^\circ_B(C_0/N)\\
&\ge\Fixl_{C_0}(C_0/N)\cap C^+_{C_0}(C_0/N)\cap C^\circ_{C_0}(C_0/N)\\
&=\Fixl_{C_0}(C_0/(C_0\cap N))\cap C^+_{C_0}(C_0/(C_0\cap N))\cap C^\circ_{C_0}(C_0/(C_0\cap N))\\
&=\Ann(C_0;C_0\cap N).\end{aligned}$$
If $X\subseteq N$ then $\Fixl_B(X/N)=C^+_B(X/N)=C^\circ_B(X/N)=B$ and $\Ann(C;C\cap N)=C$. Otherwise $C_0\not\leq N$ with $C_0/(C_0\cap N)$ annihilator nilpotent. Hence 
$\Ann(C_0;C_0\cap N)\not\leq N$, whence $\Ann(C;C\cap N)\not\leq N$.
\end{proof}

\begin{cor} A locally annihilator nilpotent stable $\omega$-categorical skew brace is annihilator nilpotent.\end{cor}
\begin{proof}By Lemma \ref{anncdiversozero} the upper annihilator series is strictly increasing. Since it consists of $\emptyset$-definable ideals, it must finish with $B$.\end{proof}

\begin{theo}\label{t:annnilpdef}
Let $B$ be a stable skew brace. If $C$ is any skew sub-brace of $B$, then there is an increasing sequence of definable skew sub-braces $(A_n:n<\omega)$ of $B$ such that $\Ann_n(C)\le A_n\le\Ann_n(A_k)$ for all $n\le k<\omega$. In particular, $A_k$ is annihilator nilpotent of class $k$; if $C$ is annihilator nilpotent, then $C$ is contained in an annihilator nilpotent definable skew sub-brace of $B$ of the same class.

Moreover, if $B$ is $\omega$-stable of finite Morley rank and $C$ a connected skew sub-brace, then $\overline{\Ann}(C)=\Ann_n(C)$ for some $n\in\omega$; if in addition $C$ is locally annihilator nilpotent, it is annihilator nilpotent.
\end{theo}
\begin{proof}
We shall define inductively two decreasing chains of definable multiplicative and additive subgroups
$$B^\circ=G_0\ge G_1\ge\cdots\quad\text{and}\quad\quad B^+=H_0\ge H_1\ge\cdots$$ and an increasing chain of definable skew sub-braces
$$\{0\}=A_0\le A_1\le\cdots$$
such that for all $n>0$:\begin{enumerate}
\item $G_n\le\Stab_B(A_n)\cap N^\circ_B(A_n)$.
\item $H_n\le N^+_B(A_n)$.
\item $C\subseteq G_n\cap H_n$.
\item $A_n\subseteq\Fixr_{H_n}(G_n/A_{n-1})\cap C_{G_n}^\circ(G_n/A_{n-1})\cap C_{H_n}^+(H_n/A_{n-1})$.
\item $A_n\cap C=\Ann_n(C)$.\end{enumerate}
Moreover $A_i\le\Ann_i(A_n)$ for $i\le n$, so $A_n$ is annihilator nilpotent of class $\le n$.\smallskip

We put $G_0=H_0=B$ and $A_0=\{0\}$.
Suppose we have defined the sequences up to $n$. We put 
$$\begin{aligned}
G_{n+1}&=C^\circ_{G_n}(\Ann_{n+1}(C)/A_n)\cap\Fixl_{G_n}(\Ann_{n+1}(C)/A_n),\\
H_{n+1}&=C^+_{H_n}(\Ann_{n+1}(C)/A_n),\quad\mbox{and}\\
A_{n+1}&=\Fixr_{H_{n+1}}(G_{n+1}/A_n)\cap C_{G_{n+1}}^\circ(G_{n+1}/A_n)\cap C_{H_{n+1}}^+(H_{n+1}/A_n).\end{aligned}$$
Then $G_{n+1}\le G_n$ and $H_{n+1}\le H_n$ are definable subgroups (multiplicative and additive, respectively) by Fact \ref{f:icc}; since $C\ast\Ann_{n+1}(C)\le\Ann_n(C)\le A_n$ and $C\subseteq G_n\cap H_n$ normalises (additively and multiplicatively) and stabilises $A_n$ we have $C\le G_{n+1}\cap H_{n+1}$; moreover $G_{n+1}\le\Fixl_{G_n}(A_{n+1}/A_n)$ stabilizes $A_{n+1}$, so 1.--4.\ hold. 

By Remark \ref{r:contained} we have $A_n\subseteq A_{n+1}\subseteq G_{n+1}\cap H_{n+1}$. We put $G=C_{G_{n+1}}^\circ(G_{n+1}/A_n)$ and $H=\Fixr_{H_{n+1}}(G_{n+1}/A_n)\cap C_{H_{n+1}}^+(H_{n+1}/A_n)$. Then $G$ stabilizes and $H$ normalizes $A_n\subseteq G\cap H$ additively; moreover $G\ast H\subseteq A_n$, so $A_{n+1}=G\cap H$ is a skew sub-brace by Lemma \ref{l:ast}; clearly it is definable.

By definition of $G_{n+1}$ and $H_{n+1}$ we have
$$\Ann_{n+1}(C)\le C_{G_{n+1}}^\circ(G_{n+1}/A_n)\cap C_{H_{n+1}}^+(H_{n+1}/A_n)\cap\Fixr_{H_n+1}(G_{n+1}/A_n)=A_{n+1}.$$
On the other hand, by inductive hypothesis $A_n\cap C=\Ann_n(C)$. Hence
$$\begin{aligned}A_{n+1}\cap C&\le \Fixr_C(C/A_n)\cap C_C^\circ(C/A_n)\cap C^+_C(C/A_n)\\
&=\Fixr_C(C/(A_n\cap C))\cap C_C^\circ(C/(A_n\cap C))\cap C^+_C(C/(A_n\cap C))\\
&=\Fixr_C(C/\Ann_n(C))\cap C_C^\circ(C/\Ann_n(C))\cap C^+_C(C/\Ann_n(C))=\Ann_{n+1}(C),\end{aligned}$$
so equality and 5.\ hold.

Finally, for any $i<n$ we have $A_{i+1}\le A_n\subseteq G_n\cap H_n\subseteq G_{i+1}\cap H_{i+1}$. By induction we may assume $A_i\le\Ann_i(A_n)$, so
$$\begin{aligned}A_{i+1}&=\Fixr_{A_n}(G_{i+1}/A_i)\cap C^\circ_{A_n}(G_{i+1}/A_i)\cap C^+_{A_n}(H_{i+1}/A_i)\\
&\le\Fixr_{A_n}(A_n/\Ann_i(A_n))\cap C_{A_n}^\circ(A_n/\Ann_i(A_n))\cap C_{A_n}^+(A_n/\Ann_i(A_n))\\&=\Ann_{i+1}(A_n)\end{aligned}$$
and $A_n$ is annihilator nilpotent of class $\le n$. Thus if $C$ is annihilator nilpotent, it is contained in a definable annihilator nilpotent skew sub-brace of $B$ of the same class.

For the moreover part, assume that $B$ is $\omega$-stable of finite Morley rank. By the descending chain condition on definable subgroups, there is $n<\omega$ such that $G_n=G_i=:G$ and $H_n=H_i=:H$ for all $i\ge n$; by the ascending chain condition on connected definable subgroups, we may take $n$ such that $A_n^0=A_i^0$ for all $i\ge n$, so $A_n$ has finite index in $A_i$.

Recall that $G=G_i\le\Stab_B(A_i)$ for all $i\ge n$. So we consider the action of $G$ via $\lambda$ on the finite set $A_i/A_n$. Then the connected component $G^0$ fixes $A_i/A_n$. Moreover, $G$ also acts on $A_i/A_n$ via multiplicative conjugation, and $G^0$ centralises $A_i/A_n$. Finally, $H$ acts on $A_i/A_n$ via additive conjugation, and $H^0$ centralises $A_i/A_n$. Put $A=\bigcup_{i<\omega}A_i$, a skew sub-brace of $B$.

Since $A\subseteq G\cap H$ and both $G/G^0$ and $H/H^0$ are finite, there is $i\ge n$ such that $G^0\circ A_i=G^0\circ A$ and $H^0+A_i=H^0+A$. Note that $C\le G^0\cap H^0$ by connectedness of $C$. Then 
$$A\subseteq (G^0\circ A_i)\cap(H^0+A_i)\subseteq\Fixl_G(A/A_i)\cap C_G^\circ(A/A_i)\cap C_H^+(A/A_i),$$
that is, $A$ is annihilator nilpotent of class $i+1$, and $\overline{\Ann}(C)=\Ann_{i+1}(C)$.

Finally, if $C$ is locally annihilator nilpotent, so is $C/C\cap A_{i+1}$; by Corollary \ref{anncdiversozero} either $C\le A_{i+1}$ or 
$$A_{i+1}\cap C=\overline{\Ann}(C)=\Ann_{i+2}(C)=\Ann(C;A_{i+1}\cap C)\not\leq A_{i+1},$$
a contradiction. Thus $C$ must be annihilator nilpotent of class at most $i+1$.\end{proof}\smallskip

The condition that $C$ be connected is necessary: Consider the trivial skew brace given by a Prüfer $2$-group inverted by an involution. As a group it is locally nilpotent but not nilpotent; being a trivial brace this translates into locally annihilator nilpotent  but not annihilator nilpotent.

\subsection{Right nilpotency and socle nilpotency}
Recall that a skew brace $B$ is right nilpotent of class $n$ if the series $B=B^{(1)}$ and $B^{(k+1)}=\langle B^{(k)}\ast B\rangle_+$ of ideals terminates in $B^{(n+1)}=\{0\}$.

More generally, let $I\trianglelefteq B$ be an ideal. The {\em kernel ideal modulo $I$} of $B$, denoted by $\Kid(B;I)$, is the largest ideal contained in $\Fixl_B(B/I)$. The {\em kernel ideal sequence modulo $I$} is defined by 
$$\Kid_0(B;I)=I\quad\mbox{and}\quad\Kid_{n+1}(B;I)=\Kid(B;\Kid_n(B;I)).$$
Then $\Kid_n(B;I)$ is an ideal of $B$, and $B/I$ is right nilpotent of class $\le k$ iff $\Kid_k(B;I)=B$. In this case $B^{(n)}\le \Kid_{k+1-n}(B;I)$ for all $0\le n\le k$. We put $\overline{\Kid}(B;I)=\bigcup_{n<\omega}\Kid_n(B;I)$; if $I=\{0\}$ it is omitted.

\begin{lem} Let $B$ be a skew brace and $I\trianglelefteq B$ a definable ideal. Then $\Kid(B;I)$ is definable. More precisely,
$$\Kid(B)=\bigcap_{b\in B}\bigcap_{b'\in B}b'+\lambda_b(\Fixl_B(B/I)-b'.$$\end{lem}
\begin{proof}Let $\bar\lambda$ be the induced $\lambda$-action on $B/I$. Then $\ker\bar\lambda$ is a trivial skew sub-brace of $B/I$, and $\Fixl_B(B/I)$ is its preimage in $B$, whence a skew sub-brace of $B$. It follows that 
$$\bigcap_{b\in B}\bigcap_{b'\in B}b'+\lambda_b(\Fixl_B(B/I))-b'$$ is a definable strong left ideal contained in $\Fixl_B(B/I)$, and it clearly is maximal such. However, for any strong left ideal  $J\le\Fixl_B(B/I)$ we have $J*B\le I$, so $J$ is an ideal.\end{proof}

We now turn to socle nilpotency. Note first that we can express the {\em upper socle series modulo $I$} using left fixators.
$$\begin{aligned}\Soc(B;I)&=\Soc_1(B;I)=\Fixl_B(B/I)\cap C_B^+(B/I)\quad\mbox{and}\\
\Soc_{n+1}(B;I)&=\Fixl_B(B/\Soc_n(B;I))\cap C^+_B(B/\Soc_n(B;I)).\end{aligned}$$
We put $\overline{\Soc}(B;I)=\bigcup_{n<\omega}\Soc_n(B;I)$; if $I=\{0\}$ it is omitted.

The {\em lower socle series} is given by
$$\Delta_0(B)=B,\quad\mbox{and}\quad\Delta_{n+1}(B)=\langle\Delta_n(B)\ast B,\,[B,\Delta_n(B)]_+\rangle_+.$$
These are again ideals in $B$, and $B/I$ is socle nilpotent of class $\le k$ iff $\Soc_k(B;I)=B$ iff $\Delta_k(B)\le I$; in this case $\Delta_n(B)\le\Soc_{k-n}(B;I)$ for all $0\le n\le k$.

\begin{rem} A skew brace $B$ is socle nilpotent iff it is right nilpotent of nilpotent type, see \cite{JeVAV22x}.\end{rem}

\begin{lem}\label{l:locbddrnilp}Let $B$ be a skew brace which is locally right/socle nilpotent of class $\le k$. Then $B$ is right/socle nilpotent of class $\le k$.\end{lem}
\begin{proof}
$B^{(k)}=\bigcup_{C\le B\text{ fin.\ gen.}}C^{(k)}=\{0\}$ and $\Delta_k(B)=\bigcup_{C\le B\text{ fin.\ gen.}}\Delta_k(C)=\{0\}$.\end{proof}
\begin{lem}\label{kiddiversozero} Let $B$ be a stable skew brace, $N\trianglelefteq B$ a definable skew sub-brace, and $C$ a skew sub-brace stabilising and normalising (additively and multiplicatively) $N$. If $C/(C\cap N)$ is non-zero locally right/socle nilpotent, then $\Kid(C;C\cap N)\not\le N$ (resp.\  $\Soc(C;C\cap N)\not\leq N$).\end{lem}
\begin{proof} Put $M=\Fixl_B(C/N)$ and $H=\Fixr_B(M/N)\cap N^+_B(N)$. Then $M\ge N$ is a definable multiplicative subgroup containing $\Kid(C;N)$ by Fact \ref{f:icc} and Remark \ref{r:contained}, and $H\ge N$ is a definable additive subgroup containing $C$; both stabilise $N$, and $M\ast H\subseteq N$. So $M\cap H$ is a skew sub-brace containing $N$ as ideal with trivial quotient by Lemma \ref{l:ast}; moreover 
$$\Kid(C;N)\le\Fixl_C(C/N)\le\Fixl_{M\cap H}(C/N)=\Fixl_{M\cap H}((M\cap H)/N)=M\cap H.$$

By Fact \ref{f:icc} there is a finite set $X\subseteq C$ such that $\Fixl_B(X/N)=\Fixl_B(C/N)$ and
$$\bigcap_{c\in C}\bigcap_{c'\in C}c'+\lambda_c(\Fixl_{M\cap H}(C/N))-c'=\bigcap_{c\in X}\bigcap_{c'\in X}c'+\lambda_c(\Fixl_{M\cap H}(X/N))-c'$$
or, respectively,
$C^+_B(X/N)=C^+_B(C/N)$.

Let $C_0$ be the skew sub-brace generated by $X$. Then $\Kid(C_0;C_0\cap N)\le\Kid(C;C\cap N)$ (resp.\ $\Soc(C_0;C_0\cap N)\subseteq\Soc(C;C\cap N)$) as in the proof of Lemma \ref{anncdiversozero}. If $X\subseteq N$ then $\Kid(C;C\cap N)=C$ (resp.\ $\Soc(C;C\cap N)=C$). Otherwise $C_0\not\le N$ with $C_0/(C_0\cap N)$ right/socle nilpotent. Hence $\Kid(C_0;C_0\cap N)\not\le N$, whence $\Kid(C;C\cap N)\not\le N$ (resp.\ $\Soc(C_0;C_0\cap N)\not\le N$, implying $\Soc(C;C\cap N)\not\le N$).\end{proof}

\begin{cor} A locally right/socle nilpotent stable $\omega$-categorical skew brace is right/socle nilpotent.\end{cor}
\begin{proof}By Lemma \ref{kiddiversozero} the kernel ideal\,/\,upper socle series is strictly increasing. Since it consists of $\emptyset$-definable ideals, it must finish with $B$.\end{proof}

\begin{theo}\label{t:rnilpdef}
Let $B$ be a stable skew brace. If $C$ is any skew sub-brace of $B$, then there are increasing sequences of definable skew sub-braces $(K_n:n<\omega)$ and $(S_n:n<\omega)$ of $B$ such that 
$$\Kid_n(C)\le K_n\le\Kid_n(K_k)\quad\mbox{and}\quad \Soc_n(C)\le S_n\le\Soc_n(S_k)$$ for all $n\le k<\omega$. In particular, $K_k$ is right nilpotent of class $k$ and $S_k$ is socle nilpotent of class $k$; if $C$ is right/socle nilpotent, then $C$ is contained in a right/socle nilpotent definable skew sub-brace of $B$ of the same class.

Moreover, if $B$ is $\omega$-stable of finite Morley rank and $C$ a connected skew sub-brace, then $\overline{\Kid}(C)=\Kid_n(C)$ and $\overline{\Soc}(C)=\Soc_n(C)$ for some $n\in\omega$; if $C$ is locally right/socle nilpotent, it is right/socle nilpotent.
\end{theo}
\begin{proof}As in the proof of Theorem \ref{t:annnilpdef} we define inductively decreasing chains of definable multiplicative and additive subgroups
$$\begin{aligned}B^\circ&=G_0\ge G_1\ge\cdots&\text{and}\qquad B^\circ&=M_0\ge M_1\ge\cdots\\
B^+&=H_0\ge H_1\ge\cdots&\text{\ and\!}\qquad B^+&=A_0\ge A_1\ge\cdots\end{aligned}$$ and increasing chains of definable skew sub-braces
$$\{0\}=K_0\le K_1\le\cdots\quad\text{and}\quad\quad \{0\}=S_0\le S_1\le\cdots$$
such that for all $n>0$:
\renewcommand\arraystretch{1.5}
$$\begin{array}{llll}
1.&G_n\le\Stab_B(K_n)\cap N^\circ_B(K_n),&\text{and}&M_n\le\Stab_B(S_n)\cap N^\circ_B(S_n);\\
2.&H_n\le N_B^+(K_n),&\text{and}&A_n\le N^+_B(S_n);\\
3.&C\subseteq G_n\cap H_n,&\text{and}&C\subseteq M_n\cap A_n;\\
4.&K_n\subseteq\Fixl_{G_n}(H_n/K_{n-1})\cap H_n,&\text{and}
&S_n\subseteq\Fixl_{M_n}(A_n/S_{n-1})\cap C^+_{A_n}(A_n/S_{n-1});\\
5.&K_n\cap C=\Kid_n(C),&\text{and}&S_n\cap C=\Soc_n(C).\end{array}$$
Moreover $K_i\le\Kid_i(K_n)$ and $S_i\le\Soc_i(S_n)$ for $i\le n$, so $K_n$ is right nilpotent of class $\le n$ and $S_n$ is socle nilpotent of class $\le n$. Note that the conditions are the same, except for condition 4.\ which ensures that $A_n$ not only normalizes $S_n$ but centralizes $S_n$ modulo $S_{n-1}$.\smallskip

We put $G_0=H_0=M_0=A_0=B$ and $K_0=S_0=\{0\}$. Suppose we have defined the sequences up to $n$. We put  
$$\begin{aligned}
K'_{n+1}&=\Fixl_{G_n}(C/K_n),&S'_{n+1}&=\Fixl_{M_n}(C/S_n),\\
H'_{n+1}&=\Fixr_{H_n}(K'_{n+1}/K_n),&A_{n+1}&=\Fixr_{A_n}(S'_{n+1}/S_n)\\
K''_{n+1}&=K'_{n+1}\cap H'_{n+1},&&\qquad\cap C^+_{A_n}(\Soc_{n+1}(C)/S_n),\\
K_{n+1}&=\bigcap_{c\in C}\bigcap_{c'\in C}c'+\lambda_c(K''_{n+1})-c',&S''_{n+1}&=S'_{n+1}\cap C_{A_{n+1}}^+(A_{n+1}/S_n),\\
G_{n+1}&=N^\circ_{G_n}(K_{n+1})\cap \Stab_{G_n}(K_{n+1}),&S_{n+1}&=\bigcap_{c\in C}\lambda_c(S''_{n+1}),\\
H_{n+1}&=N^+_{H'_{n+1}}(K_{n+1})&M_{n+1}&=N^\circ_{M_n}(S_{n+1})\cap \Stab_{M_n}(S_{n+1}).\end{aligned}$$
Then $K'_{n+1}\le G_n$ and $S'_{n+1}\le M_n$ are definable multiplicative subgroups by Remark \ref{fixdef} containing, stabilizing and normalizing (multiplicatively) $K_n$ (resp.\ $S_n$) by Remark \ref{r:contained}; moreover $\Kid_{n+1}(C)\le K'_{n+1}$ and $\Soc_{n+1}(C)\le S'_{n+1}$. Next, $H'_{n+1}\le H_n$ and $A_{n+1}\le A_n$ are definable additive subgroups containing $C$, and containing, stabilizing and normalizing (additively) $K_n$ (resp.\ $S_n$). Hence  $K''_{n+1}\ge K_n$ and $S''_{n+1}\ge S_n$ are skew sub-braces by Lemma \ref{l:ast} with trivial quotients $K''_{n+1}/K_n$ and $S''_{n+1}/S_n$.

By Fact \ref{f:icc} the additive subgroup $K_{n+1}\le K''_{n+1}$ is definable. As $K_n$ is stabilized and normalized additively by $C$, we have $K_n\le K_{n+1}$; since $K''_{n+1}/K_n$ is trivial, $K_{n+1}$ is a skew sub-brace. Note that $C$ stabilizes and normalizes (additively) both $K_{n+1}$ and $\Kid_{n+1}(C)\le K''_{n+1}$; in particular $\Kid_{n+1}(C)\le K_{n+1}$. Similarly $S_{n+1}\ge S_n$ is a definable skew sub-brace containing $\Soc_{n+1}(C)$ stabilized by $C$ and centralized (additively) by $A_{n+1}$ (whence $C$) modulo $S_n$.

It follows that $G_{n+1}\le G_n$, $H_{n+1}\le H'_{n+1}\le H_n$ and $M_{n+1}\le M_n$ are definable subgroups containing $K_{n+1}$ (resp.\ $S_{n+1}$). 
Since $C$ stabilizes and normalizes $K_{n+1}$ additively and $K_{n+1}\ast C\subseteq K_n$, we have $C\le N^\circ_B(K_{n+1})$ by Lemma \ref{l:norm}. Thus $C\subseteq G_{n+1}\cap H_{n+1}$; similarly $C\subseteq M_{n+1}\cap A_{n+1}$. Hence 1.--4.\ hold.

We already have established $\Kid_{n+1}(C)\le K_{n+1}$ and $\Soc_{n+1}(C)\le S_{n+1}$. For the reverse inclusions, note that
$$\begin{aligned}K_{n+1}\cap C&\le\bigcap_{c\in C}\bigcap_{c'\in C}c'+\lambda_c(\Fixl_C(C/K_n))-c'\\
&=\bigcap_{c\in C}\bigcap_{c'\in C}c'+\lambda_c(\Fixl_C(C/K_n\cap C))-c'\\
&=\bigcap_{c\in C}\bigcap_{c'\in C}c'+\lambda_c(\Fixl_C(C/Kid_n(C)))-c'=\Kid_{n+1}(C),\end{aligned}$$
and
$$\begin{aligned}
S_{n+1}\cap C&\le\Fixl_C(C/S_n)\cap C^+_C(C/S_n)\\
&=\Fixl_C(C/S_n\cap C)\cap C^+_C(C/S_n\cap C)\\
&=\Fixl_C(C/\Soc_n(C))\cap C^+_C(C/\Soc_n(C))=\Soc_{n+1}(C).\end{aligned}$$
Thus equality and 5.\ hold.

Finally, for any $i<n$ we have $K_n\subseteq G_n\cap H_n\subseteq G_{i+1}\cap H_{i+1}$. Hence $K_n$ stabilises $K_i$ by 1. By induction we may assume $K_i\le\Kid_i(K_n)$, so
$$K_{i+1}\le\Fixl_{K_n}(H_{i+1}/K_i)\le\Fixl_{K_n}(K_n/\Kid_i(K_n)).$$
Since $K_{i+1}$ is stabilised and normalised by $K_n$, we have $K_{i+1}\le\Kid_{i+1}(K_n)$, so $K_n$ is right nilpotent of class $\le n$. Thus if $C$ is right nilpotent, it is contained in a definable right nilpotent sub-brace of the same class.

Similarly, assuming inductively $S_i\le\Soc_i(S_n)$ we have
$$\begin{aligned}S_{i+1}&\le\Fixl_{S_n}(A_{i+1}/K_i)\cap C^+_{S_n}(A_{n+1}/S_n)\\
&\le\Fixl_{S_n}(S_n/\Soc_i(S_n))\cap C_{S_n}^+(S_n/\Soc_i(S_n))=\Soc_{i+1}(S_n),\end{aligned}$$
so $S_n$ is socle nilpotent of class $\le n$. Thus if $C$ is socle nilpotent, it is contained in a definable socle nilpotent sub-brace of the same class.

The proof of the moreover part is the same as in the proof of Theorem \ref{t:annnilpdef}, using Corollary \ref{kiddiversozero} instead of Corollary \ref{anncdiversozero}.\end{proof}

The same example as before shows that connectedness is necessary.

\subsection{Left nilpotency}
Recall that a skew brace $B$ is {\em left nilpotent} of class $k$ if the series $B=B^1$ and $B^{n+1}=\langle B\ast B^n\rangle_+$ of left ideals terminates with $B^{k+1}=\{0\}$. It is {\em strongly left nilpotent} of class $k$ if there is a sequence $\{0\}=I_0<I_1<\cdots<I_k=B$ of strong left ideals such that $B\ast I_{n+1}\subseteq I_{n}$ for all $n<k$.

\begin{lem}\label{l:locbddlnilp}Let $B$ be a skew brace which is locally left nilpotent of class $\le n$. Then $B$ is left nilpotent of class $\le n$.\end{lem}
\begin{proof}We have
$B^n=\bigcup_{C\le B\text{ fin.\ gen.}}C^n=\{0\}$.\end{proof}

\begin{lem}\label{l:stronglyln}Let $B$ be a skew brace of nilpotent type. If $B$ is left nilpotent then $B$ is strongly left nilpotent. In fact, there is a sequence $\{0\}=I_0<\cdots<I_n=B$ of strong ideals with $B\ast I_{i+1}\subseteq I_i$ for all $i<n$, and whose quotients are additively centralized by $B$.\end{lem}
\begin{proof}Suppose $(B,+)$ is nilpotent of class $k$ and $B$ is left nilpotent of class $\ell$. Put $I_{i\ell-j+1}=Z^+_i(B)\cap (B^j+Z^+_{i-1}(B))$ for $1\le i\le k$ and $1\le j\le\ell+1$ and note that 
$$Z^+_{i+1}(B)\cap (B^{\ell+1}+Z^+_i(B))=Z^+_i(B)\cap (B^1+Z^+_{i-1}(B))=Z^+_i(B)=I_{i\ell}.$$
Then the sequence of left strong ideals
	$$B=I_{k\ell}\ge I_{k\ell-2}\ge\cdots\ge I_0=\{0\}$$
witnesses that $B$ is strongly left nilpotent of class $\le k\ell$, with additively central quotients $I_{i+1}/I_i$ for all $i<k\ell$.\end{proof}
	
We shall now link nilpotency of $G(B)$ and left nilpotency of $B$. Of course, $G(B)$ is abelian iff $B$ is a trivial brace (i.e.\ an abelian group). It has been proven in \cite[Theorem 4.8]{Cedo} that for a finite skew brace $B$ of nilpotent type, being left nilpotent is equivalent to $(B,\circ)$ being nilpotent.

\begin{theo}\label{BsemiB} Let $B$ be a skew brace. Then $B$ is left nilpotent of nilpotent type iff $G(B)$ is nilpotent.\end{theo}
\begin{proof} Suppose first that $G$ is nilpotent of class $c$, and put $I_n=Z_n(G)\cap (B,+)$ (where we consider $(B,+)$ as a subgroup of $G(B)$). Then for every $n$ the $\lambda$-action of $(B,\circ)$ on $I_{n+1}/I_n$ is trivial. It follows that $I_n$ is a strong left ideal, and $B*I_{n+1}\subseteq I_n$, so $B$ is strongly left nilpotent of class $\le c$, with additively central quotients. In particular $B$ is of nilpotent type.

Conversely, suppose that $B$ is left nilpotent of nilpotent type; we may assume that we have a sequence $\{0\}=I_0<I_1<\cdots<I_n=B$ of strong left ideals with central quotients. We consider the action of $G$ on $(B,+)$ by conjugation. Recall that in $G$
$$(a,b)(c,d)=(a+\lambda_bc,bd),$$
so $(c,d)^{-1}=(-\lambda_{d^{-1}}c,d^{-1})$, and
$$(c,d)^{-1}(a,b)(c,d)=(\lambda_{d^{-1}}(-c+a+\lambda_bc),d^{-1}bd).$$
Now $G$ stabilises the subgroup chain $(I_i)$; by \cite[Theorem 1]{Hall58} the quotient $G/C_G(B,+)$ is nilpotent, of class at most $\frac12n(n-1)$. (Actually, as the chain consists of subgroups normal in $(B,+)$, the bound is $n-1$.)

Clearly $(c,d)\in C_G(B,+)$ iff $\lambda_{d^{-1}}(-c+a+\lambda_1c)=a$ for all $a\in B$, i.e.\ $-c+a+c=\lambda_d(a)$ for all $a\in B$. But then for $(a,b)\in G$ and $(c,d)\in C_G(B,+)$ we have
$$\begin{aligned}(a,b)^{-1}(c,d)^{-1}(a,b)(c,d)&=(-\lambda_{b^{-1}}a,b^{-1})(\lambda_{d^{-1}}(-c+a+\lambda_bc),d^{-1}bd)\\
&=(-\lambda_{b^{-1}}a,b^{-1})(a+\lambda_bc-c,d^{-1}bd)\\
&=(\lambda_{b^{-1}}(-a+a+\lambda_bc-c),b^{-1}d^{-1}bd)\\
&=(\lambda_{b^{-1}}(b\ast c),b^{-1}d^{-1}bd).\end{aligned}$$
Since the $I_i$ are $\lambda$-invariant, and for $e\in I_{i+1}$ and $b\in B$ we have $b\ast e\in I_i$, after $n$ commutations we obtain an element of the form $(0,e)\in C_G(B,+)$, whence $\lambda_ex=x$ for all $x\in B$ and $e\in\ker\lambda$. However, for $e\in\ker\lambda$ and $b\in B$ we have $b^{-1}\circ e^{-1}\circ b\in\ker\lambda$, whence
$$\begin{aligned}b^{-1}\circ e^{-1}\circ b\circ e&=(b^{-1}\circ e^{-1}\circ b)+e=b^{-1}\circ(e^{-1}+b)+e\\
&=b^{-1}\circ e^{-1}-b^{-1}+0+e=b^{-1}+b^{-1}\ast e^{-1}+e^{-1}-b^{-1}+e\\
&=(b^{-1}+b^{-1}\ast e^{-1}-b^{-1})+(b^{-1}-e-b^{-1}+e).\end{aligned}$$
Since the $I_i$ are additively normal, and for $e\in I_{i+1}$ and $b\in B$ we also have $[b,e]_+\in I_i$, after $n$ further commutations we reach $0$. It follows that $G$ is nilpotent, of class at most $\frac12n(n-1)+2n=\frac12n(n+3)$.
\end{proof}
\begin{rem} To show nilpotency of $(B,\circ)$ we could have used the proof of \cite[Theorem 4.8]{Cedo}, using \cite[Theorem 1]{Hall58} instead of \cite[Theorem 4]{Hall58}. We could then have considered the $\lambda$-action of $B$ on $(B,+)$ and show that it is nilpotent. However, the proof given is more direct.\end{rem} 

It has been proven in \cite[Theorem 20]{Smok2} that the sum of two left nilpotent ideals of a brace is left nilpotent as well. We shall show an analogue for skew braces.

\begin{prop}
Let $B$ be a skew brace, and $I$ and $J$ be left nilpotent ideals of nilpotent type. Then $I+J$ is again a left nilpotent ideal of nilpotent type.
\end{prop}
\begin{proof} We consider the group $G(B)$. By Lemma \ref{idealnormal} both $G(I)$ and $G(J)$ are normal subgroups of $G(B)$; by Theorem \ref{BsemiB} they are both nilpotent, as is their product $G(I)\,G(J)=G(I+J)$. Thus $I+J$ is a left nilpotent ideal.\end{proof}

Finally, we show that in a stable skew brace a left nilpotent skew sub-brace is contained in a type-definable one.
\begin{theo}
Let $B$ be a stable skew brace. If $C$ is a (strongly) left nilpotent skew sub-brace of $B$ of class $k$, then $\bar C^+=\bar C^\circ=\bar C$ is a type-definable (strongly) left nilpotent skew sub-brace of $B$ of class $k$, and there is an increasing sequence $\{0\}=\bar I_0<\bar I_1<\cdots<\bar I_k=\bar C$ of type-definable (strong) left ideals of $\bar C$ with $\bar C\ast \bar I_{n+1}\subseteq\bar I_n$ for all $n<k$. If $B$ is $\omega$-stable, $\bar C$ and all $\bar I_n$ are definable.\end{theo}
\begin{proof}Let $\{0\}=I_0<I_1<\cdots<I_k=B$ be a chain of (strong) left ideals with $C\ast I_{n+1}\subseteq I_n$ for all $n<k$. Then all $I_n$ are weakly soluble, so $\bar I_n^+=\bar I_n^\circ=\bar I_n$ for all $n\le k$. Moreover, $\bar{N_B^+(I_n)}\le N_B^+(\bar I_n)$,  
$\bar{N_B^\circ(I_n)}\le N_B^\circ(\bar I_n)$, and $\bar{\Stab_B(I_n)}\le\Stab_B(\bar I_n)$ by Lemma \ref{l:hulls}. This shows the result.\end{proof}

\section{$\omega$-categorical skew braces}\label{omega-cat}
\subsection{Minimal ideals in $\omega$-categorical skew braces} 
\begin{theo}\label{bracegen}
Let $B$ be an $\omega$-categorical skew brace, $\bar b\in B$ a tuple of parameters, and $X$ a $\bar b$-definable set. Then the skew sub-brace generated by $X$, the (strong) left ideal generated by $X$ and the ideal generated by $X$ are all uniformly $\bar b$-definable (i.e.\ the defining formulas only depend on the formula used for $X$, using the parameter $\bar b$). 
\end{theo}
\begin{proof} Since $X$ is $\bar b$-definable, the skew sub-brace/(strong) left ideal/ideal generated by $X$ are stabilized by all automorphisms of $B$ fixing $\bar b$ pointwise. So by $\omega$-categoricity they are $\bar b$-definable. The formula defining them only depends on the type $\tp(\bar b)$.  By $\omega$-categoricity again there are only finitely many possibilities for $\tp(\bar b)$, so the result follows.
\end{proof}

Of course, if $X$ is finite, the skew sub-brace generated by $X$ is again finite, uniformly in the size of $X$.

It follows in particular from Theorem \ref{bracegen} that $\Gamma_n(B)$, $\Delta_n(B)$, $B^{(n)}$ and $B^n$ are $\emptyset$-definable for all $n<\omega$, and the sequences become stationary after finitely many steps.

\begin{cor} A definably simple $\omega$-categorical skew brace $B$ is simple.
\end{cor}
\begin{proof}If $B$ is definably simple, then $(b)=B$ for any $0\not=b\in B$. But then $B$ is simple.\end{proof}

Let $B$ be a skew brace. Let $B_S$ be the set of $b\in B$ generating a minimal ideal $(b)$. Put $B_U=\{b\in B_S:(b)\mbox{ is trivial}\}$ and $B_V=\{b\in B_S:(b)\mbox{ is non-trivial}\}$, so $B_U\cup B_V=B_S$. Let $\mathcal{S}(B)=\{(b):b\in B_S\}$ be the set of all minimal ideals of $B$, $\mathcal{U}(B)=\{(b):b\in B_U\}$ the set of all trivial minimal ideals, and $\mathcal{V}(B)=\{(b):b\in B_V\}$ the set of all non-trivial minimal ideals. Let $S(B)$, $U(B)$ and $V(B)$ be the skew sub-braces generated by those in $\mathcal{S}(B)$, $\mathcal{U}(B)$ and~$\mathcal{V}(B)$, respectively. An easy application of Zorn's lemma yields that $S(B)$, $U(B)$ and~$V(B)$ are restricted direct sums of some of the minimal ideals in $\mathcal{S}(B)$, $\mathcal{U}(B)$ and~$\mathcal{V}(B)$, respectively; in particular, $U(B)$ is trivial.

In general, these ideals of $B$ are not at all first-order, but for $\omega$-categorical skew braces the situation is better.

\begin{theo}\label{sbdef}
Let $B$ be an $\omega$-categorical skew brace. Then $B_S$, $B_U$ and $B_V$ are $\emptyset$-definable subsets of $B$, and $S(B)$, $U(B)$ and $V(B)$ are $\emptyset$-definable.  
\end{theo}
\begin{proof}
By Theorem \ref{bracegen} the ideal $(b)$ is uniformly $b$-definable; it is minimal if $(a)=(b)$ for all $0\not=a\in(b)$. Hence $B_S$ is $\emptyset$-definable. But triviality of $(b)$ is definable by the formula $\forall\, x,y\in(b)\ x+y=x\circ y$. So $B_U$ is $\emptyset$-definable, as is $B_V=B_S\setminus B_U$. Finally $S(B)$, $U(B)$ and $V(B)$ are just the ideals (or in fact additive subgroups) generated by $B_S$, $B_U$ and $B_V$, respectively, which are $\emptyset$-definable by Theorem \ref{bracegen}.
\end{proof}

Let $B$ be a skew brace. We define the {\it ascending Loewy series} of $B$ as follows: $$S_0(B)=\{0\},\quad\mbox{and}\quad S_{n+1}(B)/S_n(B)=S\big(B/S_n(B)\big).$$ 
Clearly, in an $\omega$-categorical skew brace $S_n(B)$ is $\emptyset$-definable for all $n<\omega$ and there is $n$ such that $S_n(B)=S_{n+1}(B)$, i.e.\ $B$ has finite {\em Loewy length}. 

In order to obtain a more precise description of the factors of the ascending Loewy series in $\omega$-categorical skew braces, we need the following analogue of a theorem of Remak for groups.

\begin{lem}\label{remak}
Let $B$ be a skew brace, and let $I=\bigoplus_{i\in\mathcal{I}} I_i$ be the restricted direct sum of minimal ideals $I_i$ of $B$. If $J$ is any ideal of $B$ contained in $I$, then $I=J\oplus\bigoplus_{j\in\mathcal{J}}B_j$ for some $\mathcal{J}\subseteq\mathcal{I}$. 

Moreover, if all $I_i$ are non-trivial, then $J$ is actually the direct sum of certain of the $I_i$.
\end{lem}
\begin{proof}
If $I=J$, the result is obvious. Assume $I\neq J$, so there is some $j\in\mathcal{I}$ with $J\cap I_j=\{0\}$. By Zorn's Lemma there is a maximal subset $\mathcal{J}\subset\mathcal{I}$ such that $J\cap\bigoplus_{j\in\mathcal{J}} I_j=\{0\}$. Then $I_i\subseteq J\oplus \bigoplus_{j\in\mathcal{J}} I_j$ for all $i\in\mathcal{I}$, whence $I=J\oplus \bigoplus_{j\in\mathcal{J}} I_j$. This proves the first half of the statement. 

Now assume that each $I_i$ is non-trivial. We can clearly factor out any $I_i$ contained in $J$, assuming consequently that $J\cap I_i=\{0\}$ for all $i\in\mathcal{I}$. It follows that, for each $i\in\mathcal{I}$, $$[J,I_i]_+=[J,I_i]_\circ=J\ast I_i=\{0\}.$$ Therefore $$J\subseteq Z(I,+)\cap Z(I,\circ)\cap\Ker(\lambda_I)=\Ann(I),$$ where $\lambda_I$ is the obvious  restriction of $\lambda$ to the skew brace $I$, so $J$ is trivial. By the first part of the proof $I=J\oplus M$, where $M=\bigoplus_{j\in\mathcal{J}} I_j$. Then 
$$J\simeq (J\oplus M)/M\simeq\bigoplus_{i\in\mathcal{I}}I_i/\bigoplus_{j\in\mathcal{J}}I_j\simeq\bigoplus_{i\in\mathcal{I}\setminus\mathcal{J}}I_i$$
is a direct sum of non-trivial skew braces, and consequently $J\subseteq M$. Hence~\hbox{$J=\{0\}$} and the statement is proved.
\end{proof}

\begin{cor}
Let $B$ be a skew brace. Then $S(B)$ is the direct sum of $U(B)$ and $V(B)$.
\end{cor}
\begin{proof} $U(B)\cap V(B)$ is simultaneously trivial and non-trivial, whence equals $\{0\}$.\end{proof}

\begin{cor}\label{cornumber}
Let $B$ be a skew brace, and let $\bigoplus_{i\in\mathcal{I}}I_i=\bigoplus_{j\in\mathcal{J}}J_j$ be direct sums of non-trivial minimal ideals $I_i,J_j$ of $B$. Then there is a bijection $\sigma:\mathcal{I}\to\mathcal{J}$ such that $I_i=J_{\sigma(i)}$.
\end{cor}
\begin{proof}
By Lemma \ref{remak}, for each $i\le m$ there is $j=\sigma(i)\le n$ such that $I_i=J_{\sigma(i)}$ (since $I_i$ is minimal, it cannot be the sum of two or more $J_j$), and clearly $\sigma$ must be injective. As
$$\bigoplus_{j\in\mathcal{J}}J_j=\bigoplus_{i\in\mathcal{I}}I_i=\bigoplus_{i\in\mathcal{I}}J_{\sigma(i)},$$
$\sigma$ must be surjective as well.
\end{proof}

\begin{cor}\label{corloewy2}
Let $B$ be an $\omega$-categorical skew brace. Then $B$ has only finitely many non-trivial minimal ideals.
\end{cor}
\begin{proof}
Choose $n\in\omega$, and suppose $c_i\in B_V$ for $i<n$ are such that $(c_i)\not=(c_j)$ for $i\not=j$. Then the $(c_i)$ are distinct non-trivial minimal ideals in $B$, and by Lemma \ref{remak} we have $\bigoplus_{i<k}(c_i)\cap (c_k)=\{0\}$ for all $k<n$, so their sum is direct. Put $b=\sum_{i<n} c_i$, and let $J$ be the ideal generated by $b$. Then $J\leq \bigoplus_{i<n}(c_i)$, and since $b$ cannot lie in any proper sub-sum, we must have equality. Moreover, by Corollary \ref{cornumber} the ideals $(c_i)$ for $i<n$ are uniquely determined by $b$ (up to permutation), as is $n$. In particular, tp$(b/\emptyset)$ implies that $(b)$ is the sum of $n$ elements of $B_V$ which generate distinct ideals.

By $\omega$-categoricity, $n$ must be bounded. Hence there can only be finitely many distinct non-trivial minimal ideals.
\end{proof}

In order to ensure the existence of minimal ideals, we have to add a tameness condition, NSOP.

\begin{lem}\label{lemmazoc}
Let $B$ be an $\omega$-categorical NSOP skew brace. If $I$ is a non-zero ideal of $B$, then there is a minimal ideal of $B$ contained in $I$.
\end{lem}
\begin{proof}By $\omega$-categoricity there is a formula $\phi(x,y)$ expressing $(x)<(y)$. Clearly $\phi$ defines a partial pre-order; as it cannot have infinite chains, there must be minimal elements, generating minimal ideals.
\end{proof}

\begin{theo}\label{loewy2}
Let $B$ be an $\omega$-categorical NSOP skew brace. Then there is $n\in\omega$ such that $S_n(B)=B$.
\end{theo}
\begin{proof}
Let $n$ be the Loewy length of $B$. If $S_n(B)<B$ then the quotient $B/S_n(B)$ contains a minimal ideal by Lemma \ref{lemmazoc}, contradicting $S_{n+1}(B)=S_n(B)$.
\end{proof}

\begin{cor}
Let $B$ be a $\omega$-categorical locally annihilator nilpotent NSOP skew brace. Then $B$ is annihilator nilpotent.\end{cor}
\begin{proof}
It follows from \cite[Theorem 4.6]{BEFPT23} (and induction) that $S_n(B)\subseteq\Ann_n(B)$ for every $n\in\omega$. Then $B=S_m(B)=\Ann_m(B)$ for some $m\in\omega$ by Theorem \ref{loewy2}.
\end{proof}

\begin{theo}\label{compositionseries}
Let $B$ be an $\omega$-categorical NSOP skew brace. Then there is a finite chain of $\emptyset$-definable skew sub-braces $$\{0\}=A_0\triangleleft A_1\triangleleft\cdots\triangleleft A_n=B$$ such that:
\begin{enumerate}[label=$(\arabic*)$]
    \item $A_i$ is an ideal of $A_{i+1}$, for each $0\leq i<n$;
    \item $A_{i+1}/A_i$ is the restricted direct sum of simple skew braces, which are either all trivial or all non-trivial; in the latter case the direct sum is finite.
\end{enumerate}
\end{theo}
\begin{proof}
By $\omega$-categoricity, there is a maximal chain $\{0\}=A_0\triangleleft A_1\triangleleft\cdots\triangleleft A_n=B$ such that $A_i$ is an $\emptyset$-definable ideal in $A_{i+1}$ for all $i<n$. Then any quotient $Q_i=A_{i+1}/A_i$ for $i<n$ is an $\omega$-categorical skew brace without any $\emptyset$-definable ideal, since otherwise we could refine the chain. Hence $U(Q_i)=Q_i$ or $V(Q_i)=Q_i$ by Theorem \ref{sbdef} and Lemma \ref{lemmazoc}. Thus $Q_i$ is the restricted direct sum of minimal ideals. But if $I$ is a minimal ideal of $Q_i$, by Lemma \ref{remak} applied to $Q_i$ there is an ideal $J$ of $Q_i$ such that $Q_i=I\oplus J$, so any ideal $I'$ of $I$ is an ideal of $Q_i$. By minimality $I'=I$, and $I$ is simple.

Finally, if $Q_i=U(Q_i)$, then $Q_i$ is trivial; if $Q_i=V(Q_i)$, then $Q_i$ is a direct sum of finitely many non-trivial simple braces by Corollary \ref{corloewy2}.
\end{proof}

Following \cite{Smok3}, we say that a skew brace $B$ is {\it strongly prime} if any kind of $\ast$-product of non-zero ideals is non-zero. An ideal $I$ of $B$ is {\it strongly prime} if $B/I$ is strongly prime as a skew brace. It has been proven in \cite{Smok3} that if $B$ is a finite brace, then the intersection $R$ of all strongly prime ideals of $B$ is the weakly soluble radical of $B$. Our next aim is to prove that this also holds for $\omega$-categorical NSOP skew braces.

Let $b$ be any element of a skew brace $B$. Following \cite{Smok3}, we define an {\em $N$-sequence} of elements of $B$ as a sequence $(b_i)_{i<\omega}$ in $B$ such that $b_{i+1}\in(b_i)_{(\alpha_i)}$ for every $i<\omega$ and for some positive integers $0<\alpha_1<\alpha_2<\cdots$ We say that an ideal $I$ of $B$ is an {\em $N$-ideal} if every $N$-sequence of elements of $I$ reaches zero. 

\begin{lem}	Let $B$ be an $\omega$-categorical NSOP skew brace. If $I$ is an $N$-ideal of $B$, then $I$ is weakly soluble.\end{lem}
\begin{proof}
By Theorem \ref{loewy2} and Corollary \ref{corloewy2}, we may assume $I$ is a non-soluble, minimal ideal of $B$. Since $I$ is not soluble, $I_{(i)}\neq\{0\}$ for all $i<\omega$ and we can choose $0\not=b_i\in I_{(i)}$; note that $(b_i)=I$ by minimality. Then $(b_i)_{i<\omega}$ is an $N$-sequence which does not reach zero, a contradiction.
\end{proof}

\begin{theo}
Let $B$ be an $\omega$-categorical NSOP skew brace. The intersection of all strongly prime ideals of $B$ coincides with the soluble radical of $B$.
\end{theo}
\begin{proof}
Let $R$ be the intersection of all strongly prime ideals of $B$, and $S$ the weakly soluble radical of $B$. 

If $S\not\le R$ there is a strongly prime ideal $I$ and $b\in S\setminus I$ such that $(b)+I/I$ is trivial. Then $((b)+I)\ast((b)+I)/I=I$, contradicting strong primality of $I$. Hence $S\le R$.

For the other direction consider $b\in R$; we show that $(b)$ is soluble. By the previous lemma we only need to show that $(b)$ is an $N$-ideal of $B$. Suppose by contradiction this is not the case, and let $(b_i)_{i<\omega}$ be an $N$-sequence of non-zero elements of $(b)$. Let $I$ be an ideal of $B$ which is maximal with respect to not containing any of the $b_i$'s. We show that $I$ is a strongly prime ideal of $B$.

So let $J_1,\ldots,J_m$ be ideals of $B$ properly containing $I$. By maximality of $I$, there is $j\in\omega$ such that $b_k\in J=J_1\cap\ldots\cap J_m$ for every $k>j$. Let $L$ be any $\ast$-product of the $J_1,\ldots,J_m$. Then there is $h\in\omega$ such that $L\supseteq J_{(h)}$. Now, $$b_{i+1}\in(b_i)_{(\alpha_i)}\subseteq J_{(\alpha_i)}\subseteq J_{(h)}\subseteq L$$ for every $i>h,j$. Since $b_{i+1}\notin I$, the quotient $(L+I)/I$ is non-zero.

It follows that $I$ is strongly prime and $R\le I$, contradicting $b_0\in R\setminus I$. Therefore $(b)$ is an $N$-ideal, whence weakly soluble, and $R\subseteq S$.\end{proof}

\subsection{$\omega$-categorical tame skew braces}
Recall that $\omega$-categorical skew braces are in particular uniformly locally finite. We shall start by studying locally finite skew braces.

\begin{lem}\label{sylowsubgroupsideals}
Let $B$ be a locally finite skew brace. If $(B,+)$ and $(B,\circ)$ are locally nilpotent, then the Sylow subgroups of $(B,+)$ are ideals of $B$, and $B$ is the restricted direct sum of these ideals.
\end{lem}
\begin{proof}
Let $p$ be a prime and $B_p$ be the Sylow $p$-subgroup of $(B,+)$. Then $B_p$ is additively characteristic, whence normal and $\lambda$-invariant, i.e.\ a strong left ideal. Let $F$ be any finite skew sub-brace of $B$. Then $F_p=B_p\cap F$ is the Sylow $p$-subgroup of $(F,+)$. Now, $F_p$ is also a skew sub-brace and its order is precisely that of the Sylow $p$-subgroup of $(F,\circ)$. Thus $F_p$ is even normal in $(F,\circ)$ and consequently $B_p$ is multiplicatively normal in $B$, whence an ideal. The fact that $B$ is the restricted direct sum of these ideals follows easily.
\end{proof}

An important group-theoretic condition is the chain condition on centralisers (\mc). It has strong implications for nilpotency considerations.
\begin{fact}\label{factlocnilp}Let $G$ be an \mc-group.\begin{enumerate}
\item The Fitting subgroup generated by all normal nilpotent subgroups is nilpotent \cite{bludov98}.
\item If $G$ is locally nilpotent, it is hypercentral \cite{bludov98}.
\item If $G$ is locally nilpotent and periodic, it is virtually nilpotent \cite[Theorem A]{Bryant79}.
\item If $G$ is locally finite of finite exponent, it is virtually nilpotent \cite{kegel}.
\end{enumerate}
By \cite[Lemma 3.8]{Baum59}, even without \mc, if $G$ is (nilpotent $p$ of finite exponent)-by-(finite $p$), it is nilpotent. It follows that a locally nilpotent \mc-group of finite exponent is nilpotent.
\end{fact}

\begin{cor}Let $B$ be a skew brace such that $G(B)$ is \mc.
Then the ideal $F$ generated by all left nilpotent ideals of nilpotent type is again left nilpotent of nilpotent type.\end{cor}
\begin{proof} For any left nilpotent ideal $I$ of nilpotent type, the group $G(I)$ is normal in $G(B)$ by Lemma \ref{idealnormal} and nilpotent by Theorem \ref{BsemiB}, whence contained in the Fitting subgroup of $G(B)$. Hence $G(F)$ is contained in the Fitting subgroup, which is nilpotent by Fact \ref{factlocnilp}. It follows that $G(F)$ is nilpotent, so $F$ is left nilpotent of nilpotent type by Theorem \ref{BsemiB}.
\end{proof}

We shall say that a skew brace has {\em finite exponents} if both $(B,+)$ and $(B,\circ)$ have finite exponent.
\begin{theo}\label{theoleftnilp}
Let $B$ be a locally finite skew brace of finite exponents such that $G(B)$ is \mc. Then the following statements are equivalent:
\begin{enumerate}[label=$(\arabic*)$]
\item $B$ is left nilpotent of nilpotent type.
\item $G(B)$ is nilpotent.
\item $(B,+)$ and $(B,\circ)$ are locally nilpotent.
\end{enumerate}
\end{theo}
\begin{proof}
The equivalence between (1) and (2) is Theorem \ref{BsemiB}. Clearly (2) implies (3).

Assume now that (3) holds. By Lemma \ref{sylowsubgroupsideals} the Sylow $p$-subgroups $B_p$ are ideals in $B$; as $B$ has finite exponents, there are only finitely many of them. Consider a finite subset $X$ of $G(B_p)$, and let $F$ be the finite skew sub-brace generated by the coordinate elements of pairs in $X$. Then $X\subseteq G(F)$ is a finite $p$-group, whence nilpotent. So $G(B_p)$ is locally nilpotent \mc\ of finite exponent, whence nilpotent. As the $B_p$ are ideals in $B$, we have $B_p\ast B_q\subseteq B_p\cap B_q=\{0\}$ and the $\lambda$-action of $B_p$ on $B_q$ is trivial for $p\not=q$, so $G(B)$ is nilpotent.
\end{proof}

A skew brace $B$ is {\it left nil} if for every $b\in B$ there is $n\in\omega$ such that $$b\ast_n b=\underbrace{b\ast(b\ast(\ldots \ast b)\ldots))}_{n\textnormal{ times }\ast}=0.$$ 
Smoktunowicz \cite{Smok} proved that if $B$ finite and $(B,+)$ is abelian, then $B$ is left nil iff it is left nilpotent. Using the previous result, we can extend Smoktunowicz' theorem.

\begin{cor}\label{corleftnil}
Let $B$ be a locally finite brace of finite exponents such that $G(B)$ is \mc. If $B$ is left nil, then it is left nilpotent.
\end{cor}
\begin{proof}
Every finite sub-brace $F$ of $B$ is left nil, whence left nilpotent by \cite[Theorem 12]{Smok}. Thus $(F,\circ)$ is nilpotent by Theorem \ref{theoleftnilp}. So $(B,\circ)$ is locally nilpotent, and $B$ is left nilpotent by Theorem \ref{theoleftnilp} again.\end{proof}

It is well known that Fitting Theorem can be weakened assuming that only one of the subgroups is normal while the other is just subnormal: this can be easily proved by induction on the subnormal defect of the subgroup which is subnormal. As a consequence of Theorem \ref{theoleftnilp}, we have a similar result for left nilpotency of nilpotent type in locally finite skew braces of finite exponent such that $G(B)$ is \mc.
\begin{cor}\label{fittingsubnormal}
Let $B$ be a locally finite skew brace of finite exponents such that $G(B)$ is \mc. If $I$ is a left nilpotent ideal of nilpotent type and~$J$ is a left nilpotent sub-ideal of nilpotent type, then $I+J$ is left nilpotent of nilpotent type.
\end{cor}
\begin{proof} Since $I$ (resp.\ $J$) is a nilpotent normal (resp.\ nilpotent subnormal) subgroup of $(B,+)$ and of $(B,\circ)$, we have that the additive and the multiplicative group of $I+J=I\circ J$ is nilpotent, so Theorem~\ref{theoleftnilp} yields that $I+J$ is left nilpotent.\end{proof}

Recall that if $\mathfrak X$ is a property, a group is virtually $\mathfrak X$ if it has a normal subgroup of finite index which is $\mathfrak X$. We shall call a skew brace $B$ {\it virtually} $\mathfrak X$ if it has an ideal of finite index which is $\mathfrak X$. 

If $\mathfrak X$ is closed under subgroups, then having an $\mathfrak X$ subgroup and having a normal $\mathfrak X$ subgroup are the same. However, it is not clear (and in fact we think it is unlikely in general) that a skew sub-brace of finite and equal indices contains an ideal of finite index; Theorem \ref{conncpt} shows this in the $\omega$-stable or stable $\omega$-categorical context. In fact, we do not even know (and do not think) that if $A$ is a skew sub-brace such that the additive or the multiplicative index is finite, the other index must be finite as well, or even that the indices agree if they are both finite. This obviously holds if $A$ contains an ideal of finite index; the next lemma shows it holds for locally finite skew braces.

\begin{lem} Let $B$ be a locally finite skew brace, and $A$ a skew sub-brace such that the additive or the multiplicative index is finite. Then the two indices agree.\end{lem}
\begin{proof} Suppose not. Let $\bar a$ be representatives for the smaller index, and $\bar b$ with $|\bar b|>|\bar a|$ some representatives for the bigger index. Consider the finite skew sub-brace $F$ generated by $\bar a,\bar b$. Since $F$ is finite, the skew sub-brace $A\cap F$ has the same additive and multiplicative index in $F$, which is $|\bar a|$ on one hand, but also $\ge|\bar b|$ on the other, a contradiction.\end{proof}

We now turn to $\omega$-categorical stable skew braces.
\begin{fact}[\cite{baur}]\label{wcatgp} An $\omega$-categorical connected $\omega$-stable group is abelian.\end{fact}

\begin{theo} An $\omega$-categorical $\omega$-stable skew brace is virtually trivial of abelian type; an $\omega$-categorical stable skew brace is virtually left nilpotent of nilpotent type.\end{theo}
\begin{proof} The connected component $B^0$ is a definable ideal in $B$ by Theorem \ref{conncpt}. The group $G(B^0)$ is definable, connected, and $\omega$-categorical. If $B$ is $\omega$-stable, so is $G$, whence abelian by Fact \ref{wcatgp}, and $B^0$ is trivial of abelian type. If $B$ is stable, so is $G$. Therefore $G$ is connected locally finite \mc\ of finite exponent, whence nilpotent by Fact \ref{factlocnilp}; it follows that $B^0$ is left nilpotent of nilpotent type by Theorem \ref{BsemiB}.\end{proof}

\section{$\pi$-nilpotency for skew braces of nilpotent type}\label{local-nilpotency}
Let $\pi$ be a set of primes, and $\pi'$ its complement. A periodic group $G$ is called {\em $\pi$-nilpotent of class $c$} if it has a normal $\pi'$-subgroup $N$ such that $G/N$ is a nilpotent $\pi$-group of class $c$. It is clear that a locally finite group is nilpotent of class $\le c$ iff it is $p$-nilpotent of class $\le c$ for all primes $p$. In this section we shall generalize this concept to skew braces.

\subsection{Left $\pi$-nilpotency}

A skew brace $B$ of nilpotent type is called {\em left $\pi$-nilpotent of class $\le d$} if $L_d(B,B_\pi)=\{0\}$, where $B_\pi$ is the Hall $\pi$-subgroup of $(B,+)$ and $L_d(X,Y)$ is defined recursively by 
$$L_0(X,Y)=\langle Y\rangle_+\qquad\mbox{and}\qquad L_{d+1}(X,Y)=\langle X\ast L_d(X,Y)\rangle_+.$$
Note that $L_n(B,B)=B^n$.

For an equivalent point of view, let us dualize the socle construction. If $B$ is a skew brace of nilpotent type, $I$ a strong left ideal and $G$ a multiplicative subgroup, define inductively a central series of additive normal subgroups by
$$F_I^0(G)=\{0\}\quad\mbox{and}\quad F_I^{n+1}(G)=\Fixr_I(G/F_n(G))\cap C_I^+(B/F_n(G)).$$
Then all $F_I^i(G)$ are stabilized by $N_B^\circ(G)$. In particular, $F_I^i(B)$ forms a central series of strong left ideals.

\begin{lem}\label{l:lpinilp} A skew brace $B$ of nilpotent type is left $\pi$-nilpotent iff $B_\pi\le F_B^{cd}(B)$, where $c$ is the nilpotency class of $(B,+)$ and $d$ the left $\pi$-nilpotency class of $B$.\end{lem}
\begin{proof} If $B_p\le F_B^n(B)$ it is easy to see that $L_n(B,B_\pi)=\{0\}$ since $B\ast F_B^{i+1}(B)\subseteq F_B^i(B)$ for all $i<n$.

Conversely, suppose $B$ is left $\pi$-nilpotent, and consider the chain $L_i(B,B_\pi)$ of left ideals. As in Lemma \ref{l:stronglyln}, since $B$ is of nilpotent type we can refine this chain to a central chain $B_\pi=I_{cd}>\cdots>I_0=\{0\}$ of left ideals, such that
$B\ast I_{i+1}\subseteq I_i$ for all $i<cdn$. But by centrality the left ideals $I_i$ are strong, and one sees inductively that $I_i\le F_B^i(B)$.
\end{proof}

\begin{prop}
Let $B$ be a skew brace of periodic nilpotent type. Then $B$ is left nilpotent iff there is finite $d$ such that $B$ is left $p$-nilpotent of class $\le d$ for all primes $p$.
\end{prop}
\begin{proof} This was shown for finite braces in \cite[Lemma 13]{adolfo}, with a different proof.
Since $B$ is of nilpotent type, $B$ is left nilpotent iff $B\le F_B^n(B)$ for some $n$. On the other hand, as $B$ is periodic and nilpotent, $B\le F_B^n(B)$ iff $B_p\le F_B^n(B)$ for all primes $p$. The result follows.\end{proof}

By \cite[Theorem 14]{adolfo}, a finite brace is $p$-nilpotent iff its multiplicative group is $p$-nilpotent. This is still true for locally finite braces. However, in the non-locally finite case, even for braces it is not true that an additively periodic element must be multiplicatively periodic, or that an additive $p$-element (resp.\ $p'$-element) must be multiplicatively $p$ (resp.\ $p'$).

We shall call a skew brace {\em bi-periodic} if it is additively and multiplicatively periodic. A bi-periodic skew brace is {\em $\pi$-separating} if every skew sub-brace which is additively $\pi$ (resp.\ $\pi'$) is multiplicatively $\pi$ (resp.\ $\pi'$). Clearly any locally finite skew brace is bi-periodic and $\pi$-separating for all $\pi$.

\begin{lem}\label{l:eqcond}	Let $B$ be a skew brace of periodic nilpotent type for some set  $\pi$ of primes. Then the following are equivalent: 
	\begin{enumerate}[label=$(\arabic*)$]
		\item $B_{\pi'}\ast B_\pi=\{0\}$;
		\item $B_{\pi'}$ is an ideal;
		\item $G(B_{\pi'})$ is normal in $G(B)$.
	\end{enumerate}
\end{lem}
\begin{proof}
If $B_{\pi'}\ast B_\pi=\{0\}$, then for $b'\in B_{\pi'}$ and $b\in B$ there are $b_0\in B_{\pi'}$ and $b_1\in B_{\pi}$ with $b=b_0+b_1$, so
$$b'\ast b=b'\ast(b_0+b_1)=b'\ast b_0+b_0+b'\ast b_1-b_0=b'\ast b_0+b_0+0-b_0\in B_{\pi'},$$ whence $B_{\pi'}\ast B\subseteq B_{\pi'}$ and $B_{\pi'}$ is an ideal. Conversely, if $B_{\pi'}$ is an ideal, then $B_{\pi'}\ast B_\pi\subseteq B_{\pi'}\cap B_\pi=\{0\}$. Thus (1) and (2) are equivalent.

Since $B_{\pi'}$ is a strong left ideal, clearly (2) and (3) are equivalent.
\end{proof}

Note that without $\pi$-separation, $G(B_\pi)$ need not be a $\pi$-group nor  $G(B_{\pi'})$ a $\pi'$-group.

\begin{theo}\label{pnilpequivalence}
Let $B$ be a bi-periodic $\pi$-separating skew brace of nilpotent type for some set $\pi$ of primes. Then the following conditions are equivalent:
	\begin{enumerate}[label=$(\arabic*)$]
	\item $B$ is left  $\pi$-nilpotent;
	\item $B_{\pi'}$ is an ideal and $B_\pi$ is left nilpotent as a skew sub-brace;
	\item $G(B)$ is  $\pi$-nilpotent.
\end{enumerate}\end{theo}
\begin{proof} Suppose $B$ is left  $\pi$-nilpotent. Note first that $F_{B_\pi}^n(B)=B_\pi\cap F_B^n(B)$ for all $i$, and that $B_{\pi'}+F_{B_\pi}^n(B)$ is a strong left ideal, for all $n$. We shall show inductively that $B_{\pi'}$ is multiplicatively normal in $B_{\pi'}+F_{B_\pi}^n(B)$. This is clear for $n=0$, so suppose it holds for $n$. Now 
$$\begin{aligned}(B_{\pi'}+F_{B_\pi}^n(B))&\ast(B_{\pi'}+F_{B_\pi}^{n+1}(B))=\\
&=(B_{\pi'}+F_{B_\pi}^n(B))\ast B_{\pi'}+(B_{\pi'}+F_{B_\pi}^n(B))\ast F_{B_\pi}^{n+1}(B)\\&\subseteq B_{\pi'}+F_{B_\pi}^n(B).\end{aligned}$$
It follows that the strong left ideal $B_{\pi'}+F_{B_\pi}^n(B)$ is an ideal in $B_{\pi'}+F_{B_\pi}^{n+1}(B)$. So if $g\in B_{\pi'}+F_{B_\pi}^{n+1}(B)$, then, since $B_{\pi'}+F_{B_\pi}^n(B)=B_{\pi'}\circ F_{B_\pi}^n(B)$, we have
multiplicatively
$$B_{\pi'}/(B_{\pi'}\cap B_{\pi'}^g)\cong (B_{\pi'}\circ B_{\pi'}^g)/B_{\pi'}\le (B_{\pi'}\circ F_{B_\pi}^n(B))/B_{\pi'}\cong F_{B_\pi}^n(B).$$
However, by $\pi$-separation $F_{B_\pi}^n(B)\le B_\pi$ is multiplicatively a $\pi$-group, and $B_{\pi'}$ is multiplicatively a $\pi'$-group, which implies that $B_{\pi'}/(B_{\pi'}\cap B_{\pi'}^g)$ is trivial and $B_{\pi'}\le B_{\pi'}^g$. Considering $g^{-1}$ we get equality, so $B_{\pi'}$ is normal in $B_{\pi'}+F_{B_\pi}^{n+1}(B)$. This finishes the induction, and proves multiplicative normality of $B_{\pi'}$ in $B$. Moreover, since $$B^{(d+1)}=L_d(B_\pi,B_\pi)\le L_d(B,B_\pi)=0$$ for big enough $d$, the skew sub-brace $B_\pi$ is nilpotent, so (1) implies (2).

Assume (2). Then $G(B)=G(B_{\pi'})\,G(B_\pi)$; the first factor is a normal  $\pi'$-subgroup, and the second a nilpotent $\pi$-subgroup by Theorem \ref{BsemiB}, so (3) holds.

Finally, assume (3), and let $N$ be the normal  $\pi'$-subgroup of $G(B)$ such that the quotient $G(B)/N$ is a nilpotent  $\pi$-group. So for every $g\in G(B)$ there is some $\pi$-number $n>0$ such that $g^n\in N$. It follows that all  $\pi'$-elements are already in $N$. In particular $G(B_{\pi'})\le N$. But $G(B_\pi)$ is a $\pi$-group and intersects $N$ trivially. It follows that $N=G(B_{\pi'})$, so the latter is normal, and $B_{\pi'}$ is multiplicatively normal in $B$. Thus $B_{\pi'}\ast B_\pi=\{0\}$ by Lemma \ref{l:eqcond}. Moreover $G(B_\pi)\cong G(B)/G(B_{\pi'})$ is nilpotent, and $B_\pi$ is a nilpotent brace by Theorem \ref{BsemiB}.

Now the equality $(a\circ b)\ast c=a\ast(b\ast c)+a\ast c+b\ast c$ implies inductively
$$L_n(B,B_\pi)=L_n(B_{\pi'}\circ B_\pi,B_\pi)=L_n(B_{\pi'},B_\pi)+L_n(B_\pi,B_\pi)=L_n(B_\pi,B_\pi)$$ for every $n\geq0$.
In particular $L_d(B,B_\pi)=L_d(B_\pi,B_\pi)=0$ for large enough $d$ by nilpotency of $B_\pi$, showing (1).
\end{proof}

It was shown by Smoktunowicz \cite[Theorem 20]{Smok2} that the sum of finitely many left $p$-nilpotent ideals of a brace is left $p$-nilpotent. We shall see that the same holds for $\pi$-nilpotency in skew braces of periodic nilpotent type.

\begin{theo}\label{fittingpnilp}
Let $B$ a skew brace, and $\pi$ a set of primes. If $I$ and $J$ are left $\pi$-nilpotent ideals of nilpotent type, then $I+J$ is left $\pi$-nilpotent of nilpotent type.
\end{theo}
\begin{proof} Clearly the sum of two ideals of nilpotent type is again an ideal of nilpotent type, so we may assume that $B=I+J$ is of nilpotent type. Then by left $\pi$-nilpotency of $I$ and $J$ there are $k$ and $\ell$ such that
	$$\begin{aligned}\{0\}&=F_{I_\pi}^0(I)<F_{I_\pi}^1(I)<\cdots<F_{I_\pi}^k(I)=I_\pi\quad\mbox{and}\\ \{0\}&=F_{J_\pi}^0(J)<F_{J_\pi}^1(J)<\cdots<F_{J_\pi}^\ell(J)=J_\pi.\end{aligned}$$
Clearly the two series consist of strong left ideals of $B$.
Put $F_{I_\pi}^{k+1}(I)=F_{J_\pi}^{\ell+1}(J)=I_\pi+J_\pi$, and 
$$H_i=\sum_{j=0}^i\big(F_{I_\pi}^j(I)\cap F_{J_\pi}^{i-j}(J)\big).$$
Since $a\ast(b+c)=a\ast b +b+a\ast c-b$, by additive normality of the $F_{I_\pi}^i(I)$ and $F_{J_\pi}^i(J)$ we get $I\ast H_{i+1}\subseteq H_i$ and $J\ast H_{i+1}\subseteq H_i$ for all $i<2n+1$. And as $(a\circ b)\ast c=a\ast(b\ast c)+a\ast c+b\ast c$
we also have $(I\circ J)\ast H_{i+1}\subseteq H_i$ for all $i<2n+1$. It follows that $I+J$ is left $p$-nilpotent of class at most $\ell+k+1$.
\end{proof}

\begin{rem} It follows that in a skew brace the sum of all left $\pi$-nilpotent ideals of nilpotent type is locally left $\pi$-nilpotent of nilpotent type.\end{rem}

\begin{cor} Let $B$ be a skew brace, and $\pi$ a set of primes. The sum of a left $\pi$-nilpotent ideal $I$ of  nilpotent type and a left $\pi$-nilpotent sub-ideal $J$ of nilpotent type is again left $\pi$-nilpotent of nilpotent type.\end{cor}
\begin{proof} Without loss of generality, we may assume $B=I+J$.
The proof is by induction on the sub-ideal defect $n$ of $J$.
If $n=1$, this is just Theorem \ref{fittingpnilp}. So suppose $n>1$ and let $H$ be an ideal of $B$ containing $J$ as a sub-ideal of defect $n-1$.
Clearly, $H\cap I$ is a left $\pi$-nilpotent ideal of nilpotent type in $H$, so by induction $(H\cap I)+J$ is left $\pi$-nilpotent of nilpotent type.
But $J\le H$ and $I+J=B$, so $H=(H\cap I)+J$ and the statement follows from Theorem \ref{fittingpnilp}.\end{proof}

\begin{theo}
Let $B$ be an $\omega$-categorical skew brace such that $G(B)$ is \mc, and let $\pi$ be a set of primes. Then the sum of all left $\pi$-nilpotent ideals of nilpotent type is left $\pi$-nilpotent of nilpotent type and definable.
\end{theo}
\begin{proof}By $\omega$-categoricity, the set $R$ of all elements generating a left $\pi$-nilpotent ideal of nilpotent type is definable. It is an ideal, and additively locally nilpotent, whence nilpotent by Fact \ref{factlocnilp}. Since the exponent is finite, the Hall subgroups $R_\pi$ and $R_{\pi'}$ are definable strong left ideals of $B$.
Moreover, $R_{\pi'}$ is generated by multiplicatively normal subgroups, so $R_{\pi'}$ is actually an ideal in $B$. It follows that $G(R_{\pi'})$ is a normal $\pi'$-subgroup of $G(R)$, and has $G(R_\pi)$ as complement. Moreover, $G(R_\pi)$ is locally nilpotent, whence nilpotent again by Fact \ref{factlocnilp}. It follows that $R$ is left $\pi$-nilpotent of nilpotent type, and contains all left $\pi$-nilpotent ideals of nilpotent type.\end{proof}

\subsection{Right $\pi$-nilpotency}
A skew brace $B$ of nilpotent type is called {\em right $\pi$-nilpotent of class $\le d$} if $R_d(B_\pi,B)=\{0\}$, where $R_n(X,Y)$ is defined recursively by 
$$R_0(X,Y)=\langle X\rangle_+\qquad\mbox{and}\qquad R_{n+1}(X,Y)=\langle R_n(X,Y)\ast Y\rangle_+.$$
Note that $R_n(B,B)=B^{(n)}$.
\begin{lem}A skew brace $B$ of nilpotent type is right $\pi$-nilpotent iff $B_\pi\le\Soc_n(B)$ for some $n<\omega$.\end{lem}
\begin{proof} This is similar to Lemma \ref{l:lpinilp}. Note that $n\le cd$, where $c$ is the nilpotency class of $(B,+)$ and $d$ the right $\pi$-nilpotency class of $B$.\end{proof}

\begin{cor}A skew brace $B$ of periodic nilpotent type is right nilpotent of class $\le d$ iff it is  right $p$-nilpotent of class $\le d$ for all primes $p$.\end{cor}
\begin{proof}Since $(B,+)$ is periodic and nilpotent, $B\le\Soc_n(B)$ iff $B_p\le\Soc_n(B)$ nilpotent for all primes $p$.\end{proof}

\begin{lem}\label{l:leftright}Let $B$ be a skew brace whose multiplicative group is abelian. Then $\Soc_i(B)=F_B^i(B)$ for all $i<\omega$. In particular $B$ is left ($\pi$-) nilpotent iff it is right ($\pi$-) nilpotent.\end{lem}
\begin{proof}By induction on $i$, the case $i=0$ being trivial. So assume it is true for $i$; dividing by $\Soc_i(B)=F_B^i(B)$ we may assume $i=1$. But then
$$\Soc_1(B)=\{a\in Z^+(B):\forall b\ a\circ b=a+b\}=\{a\in Z^+(B):\forall b\ b\circ a=b+a\}=F_B^1(B).\qedhere$$
\end{proof}

\begin{cor} Let $B$ be a skew brace of periodic nilpotent type, and $\pi$ a set of primes. Suppose $(B_\pi,\circ)$ is abelian and normal in $(B,\circ)$, and $G(B_\pi)$ is nilpotent. Then $B$ is right $\pi$-nilpotent.
\end{cor}
\begin{proof} Since $B_\pi$ is multiplicatively normal, it is an ideal in $B$. Thus $B_\pi\ast B_{\pi'}\subseteq B_\pi\cap B_{\pi'}=\{0\}$, and we show inductively that $R_i(B_\pi,B)=R_i(B_\pi,B_\pi)$. This is clear for $i=0$; suppose it holds for $i$. Then in particular $R_i(B_\pi,B_\pi)\ast B_{\pi'}=0$, so
$$\begin{aligned}R_{i+1}(B_\pi,B)&=\langle R_i(B_\pi,B)\ast B\rangle_+=\langle R_i(B_\pi,B_\pi)\ast(B_\pi+B_{\pi'})\rangle_+\\
&=\langle R_i(B_\pi,B_\pi)\ast B_\pi+0\rangle_+
=R_{i+1}(B_\pi,B_\pi).\end{aligned}$$
Since $G(B_\pi)$ is nilpotent, $B_\pi$ is left nilpotent by Theorem \ref{BsemiB}, whence right nilpotent by Lemma \ref{l:leftright}, so $R_d(B_\pi,B)=R_d(B_\pi,B_\pi)=0$ for sufficiently large $d$.
\end{proof}
\begin{rem}The hypotheses are in particular satisfied if $\pi$ consists of a single prime $p$, $(B_p,\circ)$ is abelian and normal, and $G(B_\pi)$ is a $p$-group of finite exponent which is \mc.\end{rem}

\begin{flushleft}
\rule{8cm}{0.4pt}\\
\end{flushleft}

{
\sloppy
\noindent
Maria Ferrara

\noindent
Dipartimento di Matematica e Fisica

\noindent
Università degli Studi della Campania  ``Luigi Vanvitelli''

\noindent
viale Lincoln 5, Caserta (Italy)

\noindent
e-mail: maria.ferrara1@unicampania.it
}

\bigskip
\bigskip
{
	\sloppy
	\noindent
	Moreno Invitti
	
	\noindent 
	Universit\'e Claude Bernard Lyon 1, CNRS
	
	\noindent
	Institut Camille Jordan UMR 5208,
	
	\noindent
	21 avenue Claude Bernard
	
	\noindent
	69622 Villeurbanne Cedex (France)
	
	\noindent
	e-mail: invitti@math.univ-lyon1.fr

}

\bigskip
\bigskip

{
\sloppy
\noindent
Marco Trombetti

\noindent 
Dipartimento di Matematica e Applicazioni ``Renato Caccioppoli''

\noindent
Università degli Studi di Napoli Federico II

\noindent
Complesso Universitario Monte S. Angelo

\noindent
Via Cintia, Napoli (Italy)

\noindent
e-mail: marco.trombetti@unina.it 

}

\bigskip
\bigskip

{
\sloppy
\noindent
Frank O. Wagner

\noindent 
Universit\'e Claude Bernard Lyon 1, CNRS

\noindent
Institut Camille Jordan UMR 5208,

\noindent
21 avenue Claude Bernard

\noindent
69622 Villeurbanne Cedex (France)

\noindent
e-mail: wagner@math.univ-lyon1.fr

}

\end{document}